\documentclass{article}
\usepackage{amsmath,amstext, amsthm, amssymb, dsfont,amscd,pslatex,fourier,color}
\usepackage[all,poly]{xy}
\usepackage[numbers]{natbib}

\usepackage[pdfstartview=FitH,
            colorlinks,
           linkcolor=reference,
            citecolor=citation,
            urlcolor=e-mail]{hyperref}
\definecolor{e-mail}{rgb}{0,.40,.80}
\definecolor{reference}{rgb}{.20,.60,.22}
\definecolor{citation}{rgb}{0,.40,.80}

\usepackage[totalheight=8.9in, totalwidth=6.1in,paperwidth=6.7in,paperheight=10.4in]{geometry}

\RequirePackage{color}

\date{}

\definecolor{todo}{rgb}{1,0,0}
\definecolor{answer}{rgb}{0,0,1}

\newcommand{\rd}[1]{\textcolor{answer}{#1}}

\title{Tannakian categories with semigroup actions\thanks{This work was partially supported by the NSF grants CCF-0952591 and DMS-1413859.}}

\author{Alexey Ovchinnikov  and Michael Wibmer \vspace{0.05in}  \\ \small 
CUNY Queens College, Department of Mathematics,
65-30 Kissena Blvd, Queens, NY 11367, USA \\ \small
CUNY Graduate Center, Department of Mathematics, 365 Fifth Avenue,
New York, NY 10016, USA\\ \small
\href{mailto:aovchinnikov@qc.cuny.edu}{aovchinnikov@qc.cuny.edu} \vspace{0.05in} \\ \small 
RWTH Aachen, 52056, Aachen, Germany\\ \small
\href{mailto:michael.wibmer@matha.rwth-aachen.de}{michael.wibmer@matha.rwth-aachen.de}
\vspace{-0.25in}}

\newtheorem{theo}{Theorem}[section]
\newtheorem{lemma}[theo]{Lemma}
\newtheorem{prop}[theo]{Proposition}
\newtheorem{cor}[theo]{Corollary}

\theoremstyle{definition}
\newtheorem{ex}[theo]{Example}
\newtheorem{defi}[theo]{Definition}
\newtheorem{rem}[theo]{Remark}

\numberwithin{theo}{section}
\numberwithin{equation}{subsection}

\def\N{\mathbb{N}}

\def\Hom{\operatorname{Hom}}
\def\Gl{\operatorname{GL}}

\def\<{\langle}
\def\>{\rangle}

\def\ord{\operatorname{ord}}

\DeclareMathOperator{\id}{id}
\DeclareMathOperator{\Span}{span}
\DeclareMathOperator{\Ob}{\mathcal Ob}

\DeclareMathOperator{\MorC}{\mathrm Mor}
\DeclareMathOperator{\GL}{\mathrm GL}

\newcommand{\f}{\phi}

\newcommand{\Zb}{\mathbb{Z}}

\newcommand{\Qb}{\mathbb{Q}}
\newcommand{\Cb}{\mathbb{C}}

\newcommand{\C}{\mathsf{C}}
\newcommand{\D}{\mathsf{D}}
\newcommand{\vect}{\mathsf{Vect}}
\newcommand{\alg}{\mathsf{Alg}}
\newcommand{\groups}{\mathsf{Groups}}
\newcommand{\rep}{\mathsf{Rep}}
\newcommand{\hs}{^g\!}
\newcommand{\End}{\operatorname{End}}

\def\wtilde{\widetilde}

\renewcommand{\mod}{\mathsf{Mod}}
\renewcommand{\1}{\mathds{1}}

\begin{document}

\maketitle

\abstract{
Ostrowski's theorem implies that $\log(x),\log(x+1),\ldots$ are algebraically independent over $\Cb(x)$. More generally, for a linear differential or difference equation, it is an important problem to find all algebraic dependencies among a non-zero solution $y$ and particular transformations of $y$, such as derivatives of $y$ with respect to parameters, shifts of the arguments, rescaling, etc. In the present paper, we develop a theory of Tannakian categories with semigroup actions, which will be used to attack such questions in full generality, as each linear differential equation gives rise to a Tannakian category.
Deligne studied actions of braid groups on categories and obtained a finite collection of axioms that characterizes such actions to apply it to various geometric constructions. In this paper, we find a finite set of axioms that characterizes actions of semigroups that are finite free products of semigroups of the form $\N^n\times \Zb/{n_1}\Zb\times\ldots\times\Zb/{n_r}\Zb$ on Tannakian categories. This is the class of semigroups that appear in many applications.}

\section{Introduction}
It is an important problem, for a linear differential or difference equation, to find all algebraic dependencies among a non-zero solution $y$ and particular transformations of $y$, such as derivatives of $y$ with respect to parameters, shifts of the arguments, rescaling, etc. The simplest example that illustrates this is: $\log(x)$ satisfies $y'=1/x$, while it follows from Ostrowski's theorem \cite{Ostrowski} that $\log(x),\log(x+1),\ldots$ are algebraically independent over $\Cb(x)$. It turns out that this information is contained in the Galois group associated with this differential equation~\cite{DHWDependence,DiVizioHardouinWibmer:DifferenceGaloisofDifferential}, which is a difference algebraic group, that is, a subgroup of $\GL_n$ defined by a system of polynomial difference equations. 
Other important natural examples include:
\begin{itemize}
\item the Chebyshev polynomials $T_n(x)$. They are solutions of linear differential equations $$\big(1-x^2\big)y'' - xy' + n^2y = 0$$ and, in addition, satisfy the following difference (with respect to the endomorphisms $\sigma$, $\sigma_1$, and $\sigma_2$ specified below) algebraic relations (among many other ones): 
\begin{align*}
T_{n+1}&=2xT_n(x)-T_{n-1}(x), &&\sigma(n)=n+1,\\ 
T_{2n+1}(x)&=2T_{n+1}(x)T_n(x)-x, &&\sigma_1(n)=2n,\ \sigma_2(n)=n+1,\\
T_{2n}(x)&=T_n\big(2x^2-1\big), &&\sigma_1(n)=2n,\ \sigma_2(x)=2x^2-1,\\
T_n(T_m(x))&=T_{nm}(x), &&\sigma_1(x)=T_m(x),\ \sigma_2(n)=mn.
\end{align*}
\item the hypergeometric function $_2F_1(a,b,c;z)$, which is a solution of the parameterized linear differential equation
$$
z(1-z)y''+\big(c-(a+b+1)z\big)y'-aby=0,
$$ and also satisfy the following difference  algebraic relation, called the Pfaff transformation, (among many other ones): 
$$
_2F_1(a,b,c;z) = (1-z)^{-a}{_2F_1}\big(a,c-b,c;z/(z-1)\big),\quad \sigma_1(b)=c-b,\ \sigma_2(z)=z/(z-1).
$$
\item Kummer's function of the first kind (confluent hypergeometric function) $_1F_1(a;b;z)$ is a solution of 
$$
zy''+(b-z)y'-ay=0.
$$
It satisfies the difference algebraic relation
$$
e^x{_1F_1}(a;b;-z)={_1F_1}(b-a;b;z),\quad \sigma_1(z)=-z,\ \sigma_2(a)=b-a.
$$
\item  the Bessel function $J_\alpha(x)$, which is a solution of the parameterized linear differential equation $$x^2y''(x)+xy'(x)+\big(x^2-\alpha^2\big)y(x)=0,$$ also satisfies, for example,
\begin{align*}
x J_{\alpha+2}(x) &=2(\alpha+1)J_{\alpha+1}(x)-xJ_\alpha(x), && \sigma(\alpha)=\alpha+1\\
J_\alpha(-x)&=(-1)^\alpha J_\alpha(x), &&\sigma(x)=-x.
\end{align*}
\end{itemize}
In all these cases, semigroups arise as the semigroups generated by the given endomorphisms (they are not always automorphisms). The resulting semigroups in all but one case are free commutative and finitely generated, with the exception of one example of the pair of automorphisms $\sigma_1(n) = 2n$ and $\sigma_2(n)=n+1$, which generates a Baumslag--Solitar group \cite{BaumslagSolitar}. However, this action of this particular semigroup can be viewed as an action of a commutative semigroup, as in our Example~\ref{ex:commuteandnot}. In addition, we show in Example~\ref{ex:hyper}  how the classical contiguity relations for the hypergeometric functions are reflected in our Tannkian approach. The $q$-difference analogue of the hypergeometric functions studied in this framework can be found in~\cite{OvAAM}.

Moreover,
such recurrence relations are not only of interest from the point of view of analysis and special functions, but, as emphasized in~\cite{Vilenkin-Klimyk,Koornwinder}, they also appear in the representation theory of Lie groups: they are encoded in the properties of tensor products of representations, including decompositions of tensor products into the irreducible components (e.g., Clebsch--Gordan coefficients).

In the present paper, we develop a theory of Tannakian categories with semigroup actions, which will be used to attack such questions, including finding such relations, in their full generality in the future using the Galois theory of linear differential and difference equations with semigroup actions. In this approach, given a linear differential or difference equation and a semigroup $G$, one constructs a particular Tannakian category with an action of $G$. Our Theorem~\ref{theo: sTannakianmain} shows that, if such a Tannakian category has a neutral $G$-fiber functor, then this category is equivalent to the category of representations of a difference algebraic group. This group is the one that will measure the algebraic dependencies mentioned above.

In practice, the semigroup $G$ is usually infinite and, therefore, its action on a category (see Definition~\ref{def:actionmain}) is defined by infinitely many functors and commutative diagrams, which is inconvenient in applications.
However, in \cite{Deligne1997}, Deligne studied actions of  braid groups on categories and obtained a finite collection of axioms that characterizes such actions. Tannakian categories with group actions (among other things) were first introduced in \cite{Kamensky:TannakianFormalismOverFieldsWithOperators}, but the finiteness questions were not considered there, because a different kind of applications was studied. 

In the present paper, we find a finite set of axioms that characterizes actions of semigroups that are free products of  semigroups of the form $\N^n\times \Zb/{n_1}\Zb\times\ldots\times\Zb/{n_r}\Zb$ on Tannakian categories. Even if $G$ is given by a finite set of generators and relations, as in \cite[Section~1.3]{Deligne1997}, in our case, it is not sufficient just to define actions of generators of $G$ and impose the constraints corresponding to the relations (see Example~\ref{ex:counter}) -- our hexagon axiom~\ref{eq:i1i2i3} provides necessary and sufficient extra constraints, as we show in Theorem~\ref{thm:main0}. Moreover, our approach includes more actions than one might expect:  a commutative semigroup can act in a non-commutative way on a differential equation, which is commutative only up to a gauge transformation (see Example~\ref{ex:commuteandnot}). This is the first time that such a scenario has been proposed.
The main application of our result will be to finding all algebraic dependencies among the elements of orbits of solutions of linear difference and differential equations under actions of chosen semigroups. 

This application will be possible after the parameterized Galois theories of linear differential and difference equations with semigroup actions are fully developed. So far, this has been done for the simplest case of the semigroup $\N$ in~\cite{DHWDependence,DiVizioHardouinWibmer:DifferenceGaloisofDifferential,OW} (that is, in the case of one difference parameter). The main method used in these papers was difference parameterized Picard--Vessiot rings (which correspond to neutral difference fiber functors for Tannakian categories \cite{Moshe}) that were constructed in a particular way, which does not directly generalize to arbitrary semigroups. This motivates the new approach to the problem that we take up in the present paper.

In the case of differential Galois theory with differential parameters, constructions similar to those mentioned above were used in~\cite{Wibmer:existence} to construct Picard--Vessiot extensions with one differential parameter. However, there were obstacles to generalizing this particular construction to several differential parameters as well. 
Such difficulties have recently been overcome in~\cite{GGO} by introducing actions of Lie rings on Tannakian categories (first appeared as differential tensor and Tannakian categories for one derivation~\cite{difftann,OvchTannakian,Moshe} and several commuting derivations~\cite{diffreductive}) and applying  geometric arguments to the constructions from~\cite{Deligne:categoriestannakien} to construct Picard--Vessiot extensions for several differential parameters (not necessarily commuting) under assumptions that are most practical for applications. The authors expect that the results of the present paper on actions of semigroups (instead of Lie rings) on Tannakian categories will lead to a construction of Picard--Vessiot rings with semigroup actions (that is, with several difference parameters, not necessarily commuting) with immediate practical applications in the nearest future. This includes the problem of difference isomonodromy~\cite{OvAAM}, which awaits the full development of the Picard--Vessiot theory with semigroup actions.

The paper is organized as follows. We give an overview of the constructions from difference algebra that we use in the paper in Section~\ref{sec:differencealgebra}. This is followed by Section~\ref{subsec: Basic constructions}, in which we recall difference algebraic groups and the basic constructions from their representation theory. Section~\ref{sec: Difference Tannakian categories} contains a brief review of Tannakian categories in Section~\ref{sec:revtann}, followed by Section~\ref{sec:semigroupaction}, containing an introduction to semigroup actions on categories and our main technical tool, Theorem~\ref{thm:main0}. Semigroup actions on tensor categories are described in Section~\ref{sec:semigrouptensor}, which is followed by our main result, Theorem~\ref{theo: sTannakianmain}, in Section~\ref{sec:semigrouptannakian}. We conclude with Section~\ref{sec:extra}, in which we give a representation theoretic characterization of a difference group scheme being a linear difference algebraic group. 

\vspace{-0.1in}
\section{Basic Definitions}\label{sec:differencealgebra}
\subsection{Difference algebra}
In this section, we will introduce the generalization of the standard difference algebra with one and several endo- or automorphisms~\cite{Cohn:difference,Levin:difference} that we need.
Let $G$ be a semigroup. In what follows, we will assume that $G$ has an identity element, which we will denote by $e$. If $G$ and $G'$ are semigroups, $e$ and $e'$ are their identity elements, 
and $\varphi : G\to G'$ is a semigroup homomorphism, we will assume that $\varphi(e)=e'$. In what follows, the semigroups $\N = (\{0,1,2,\ldots\},+)$ and $\Zb/r\Zb = (\{0,1,\ldots,r-1\}, +\, \mathrm{mod}\, r)$, $r \geq 1$. The semigroup of ring endomorphisms of a ring $k$ is denoted by $\End(k)$.
\begin{defi}
A {\em $G$-ring ($G$-field)} is a commutative ring (field) $k$ together with a semigroup homomorphism $T_k : G \to \End(k)$. For each $g \in G$, we also call the pair $\big(k, T_k(g)\big)$ a $g$-ring (field), write $g : k\to k$ instead of $T_k(g) : k\to k$ for simplicity (and to follow the general convention in difference algebra). 
\end{defi}

\begin{ex} Let $G = \N$ and $k = \Cb(x)$. Then $$T_1(a)(x) := x+a,\ \ T_2(a)(x) = 2^ax,\ \ T_3(a)(x) := \pi^ax,\ \ \text{and}\ \  T_4(a)(x) := x^{(2^a)},\quad a \in \N,$$ induce homomorphisms $T_1$, $T_2$, $T_3$, and $T_4$ from $G$ to $\End(k)$. Note that $T_1$ and $T_3$ also induce a homomorphism $\N*\N \to \End(k)$, where $*$ denotes the free product of semigroups, and $T_2$ and $T_3$ induce a homomorphism $\N\times\N \to \End(k)$.
\end{ex}

\begin{defi} A {\it morphism} of two $G$-rings $(R, T_R)$ and $(S, T_S)$ is a ring homomorphsim $\varphi : R \to S$ such that, for all $g \in G$, $\varphi\circ T_R(g) = T_S(g)\circ\varphi$.
\end{defi}

Let $k$ be a $G$-field.
\begin{defi} A $k$-$G$-{\it algebra} is a $k$-algebra $R$ such that $R$ is a $G$-ring and  $k\to R$ is a morphism of $G$-rings.  
\end{defi}

A morphism of $k$-$G$-algebras is a morphism of $k$-algebras that is a morphism of $G$-rings. The category of $k$-$G$-algebras is denoted by $k$-$G$-$\alg$.

\begin{defi} A $k$-$G$-algebra $(R,T_R)$ is called {\it finitely generated} if there exists a finite set $S\subset R$ such that $R$ is generated by the set $\{T(g)(s)\:|\: g\in G,\, s\in S\}$.
\end{defi}

The ring of $G$-polynomials with coefficients in $k$ in $G$-indeterminates $y_1,\ldots,y_n$ is the ring  \vspace{-0.05in} $$k\{y_1,\ldots,y_n\}_G := k\big[y_{i,g}: g \in G,1\le i \le n\big]$$ (here, $y_{i,e} = y_i$, $1\le i\le n$), with the $G$-structure given by \vspace{-0.05in}
$$h(y_{i,g}) := y_{i,hg},\quad g,h\in G,\ \ 1\le i \le n.\vspace{-0.17in}$$
\subsection{Difference algebraic groups and their representations}\label{subsec: Basic constructions}

Let $G$ be a semigroup and $k$ be a $G$-field.  In this section, we will introduce group $k$-$G$-schemes and their representations, followed by the basic constructions with the latter in Section~\ref{sec:constructions}. This is a straightforward but important generalization of difference algebraic groups studied in~\cite[Appendix]{DiVizioHardouinWibmer:DifferenceGaloisofDifferential} and \cite[Section~4.1]{Kamensky:TannakianFormalismOverFieldsWithOperators} (see also the references given in these papers).

\begin{defi}
A {\em group $k$-$G$-scheme} $H$ is a functor from the category of $k$-$G$-algebras to the category of groups that is representable. A group $k$-$G$-scheme $H$ is called a {\it $G$-algebraic group} if the $k$-$G$-algebra that represents $H$ is finitely generated.
\end{defi}
If $H$ is a group $k$-$G$-scheme, the $k$-$G$-algebra that represents $H$ is denoted by $k\{H\}$. A morphism of group $k$-$G$-schemes is a morphism of functors. If $\f\colon H\to H'$ is a morphism of group $k$-$G$-scheme, the dual morphism is denoted by $\f^*\colon k\{H'\}\to k\{H\}$.

\begin{rem}The category of group $k$-$G$-schemes is anti--equivalent to the category of $k$-$G$-Hopf-algebras, which are $k$-Hopf algebras such that all structure homomorphisms commute with $T(g)$, $g \in G$.
\end{rem}

Let $k$ be a $G$-field and $H$ a group $k$-$G$-scheme (similar for $k$-$g$-scheme for $g \in G$). 
\begin{defi}
A {\em representation} of $H$ is a pair $(V,\f)$ comprising a finite-dimensional $k$-vector space $V$ and a morphism $\f\colon H\to \Gl(V)$ of group $k$-$G$-schemes.
\end{defi}
Here $\Gl(V)$ is the functor that associates to a $k$-$G$-algebra $R$ the group of all $R$-linear automorphisms of $V\otimes_k R$. It is represented by the $k$-$G$-algebra $k\big\{x_{11},\ldots,x_{nn},1/\det(x_{ij})\big\}_G$, where $n = \dim V$.
We will often omit $\f$ from the notation. 
\begin{defi}
A \emph{morphism $(V,\f)\to \big(V',\f'\big)$ of representations of $H$} is a $k$-linear map $f\colon V\to V'$ that is $H$-equivariant, i.e.,
\[\xymatrix{
 V\otimes_k R \ar^{f\otimes R}[r] \ar_{\f(h)}[d] & V'\otimes_k R \ar^{\f'(h)}[d] \\
 V\otimes_k R \ar^{f\otimes R}[r] & V'\otimes_k R
}
\]
commutes for every $h\in H(R)$ and any $k$-$G$-algebra $R$.
\end{defi}
The resulting category is denoted by $\rep(H)$.

\begin{rem}
$\rep(H)$ is equivalent to the category of finite-dimensional comodules over $k\{H\}$.
\end{rem}

\subsection{Constructions with representations}\label{sec:constructions}
\subsubsection{Basic constructions}
There are several basic constructions one can perform with representations, which we will now recall:
\begin{itemize}
\item A $k$-sub-vector space $W$ of a representation $V$ of $H$ is called a subrepresentation of $V$ if it is stable under $H$, i.e., $h(W\otimes_k R)\subset W\otimes_k R$ for every $h\in H(R)$ and any $k$-$G$-algebra $R$. Then $W$ itself is a representation of $H$, and the quotient $V/W$ is naturally a representation of $H$.
\item
If $V$ and $W$ are representations of $H$, then the tensor product $V\otimes_k W$ is a representation of $H$ via
$$(V\otimes_k W)\otimes_k R\simeq (V\otimes_k R)\otimes_R(W\otimes_k R)\xrightarrow{h\otimes h}(V\otimes_k R)\otimes_R(W\otimes_k R)\simeq (V\otimes_k W)\otimes_k R$$ for $h\in H(R)$. \item Similarly, the direct sum $V\oplus W$ is naturally a representation of $H$.
\item
The representation of $H$ consisting of $k$ as a $k$-vector space and the trivial action of $H$ is denoted by $\1$.
\item
If $V$ and $W$ are representations of $H$, then the $k$-vector space $\Hom_k(V,W)$ of $k$-linear maps from $V$ to $W$ is a representation of $H$: For any $k$-$G$-algebra $R$, $h\in H(R)$ and $\varphi\in\Hom_k(V,W)\otimes_k R\cong\Hom_R(V\otimes_k R, W\otimes_k R)$ we define $h(\varphi)\in\Hom_k(V,W)\otimes_k R$ as the unique $R$-linear map such that
\[\xymatrix{
 V\otimes_k R \ar^{\varphi}[r] \ar_{h}[d] & W\otimes_k R \ar^{h}[d] \\
 V\otimes_k R \ar^{h(\varphi)}[r] & W\otimes_k R
}
\]
commutes, i.e., $$h(\varphi)=h\circ \varphi\circ h^{-1}.$$ In particular, if $V$ is a representation of $H$, the dual vector space $V^\vee=\Hom_k(V,k)=\Hom_k(V,\1)$ is a representation of $H$.
\end{itemize}

\subsubsection{Semigroup action}\label{subsec:action}
The above constructions with representations are familiar from the representation theory of algebraic groups. The following construction, however, is unique to difference algebraic groups and, in a certain sense (which will be made precise in Section \ref{sec: Difference Tannakian categories}), is sufficient to characterize categories of representations of difference algebraic groups.
Let $(V,\f)$ be a representation of $H$ and $g\in G$ and let
$$\hs V=V\otimes_k k$$ be the $k$-vector space obtained from $V$ by base extension via $g\colon k\to k$. A similar notation will be adopted for other objects: if $X$ is some object over $k$, then $\hs X$ denotes the object obtained by base extension via $g\colon k\to k$. There is a canonical morphism of group $k$-$G$-schemes $$g\colon \Gl(V)\to\Gl\big(\hs V\big)$$
given by associating to an $R$-linear automorphism $h\colon V\otimes_k R\to V\otimes_k R$ the $R$-linear automorphism
\begin{equation} \label{eqn: RepG is scategory}
g(h)\colon \hs V\otimes_k R\simeq (V\otimes_k R)\otimes_R R\xrightarrow{h\otimes\id_R}(V\otimes_k R)\otimes_R R\simeq {\hs V}\otimes_k R.
\end{equation}
Here, the former and latter isomorphisms are given by
$$
v\otimes a\otimes r \mapsto v\otimes 1\otimes ar\quad\text{and}
\quad
v\otimes r_1\otimes r_2 \mapsto v\otimes 1\otimes g(r_1)r_2,\quad v\in V,\ a\in k,\ r,r_1,r_2 \in R,
$$
respectively,
and the tensor product $(V\otimes_k R)\otimes_R R$ is formed by using $g\colon R\to R$ on the right-hand side.
In terms of matrices, if $\underline{e}=(e_1,\ldots,e_n)$ is a basis of $V$ and $A\in\Gl_n(R)$ represents the action of $h$ on $V\otimes_k R$, i.e, $h(\underline{e})=\underline{e}A$, then, with respect to the basis $\underline{e}\otimes 1$ of ${\hs V}$, the action of $g(h)$ on ${\hs V}\otimes_k R$ is represented by $g(A)\in\Gl_n(R)$.

We can define a new representation $\big(\hs V,g(\f)\big)$ of $H$ as the composition
$$g(\f)\colon H\xrightarrow{\f}\Gl(V)\xrightarrow{g}\Gl\big(\hs V\big).$$
If $f\colon V\to W$ is a morphism of representations of $H$, then also $\hs f\colon \hs V\to\hs W$ is a morphism of representations of $H$. Thus $V\rightsquigarrow{\hs V}$ is a functor from $\rep(H)$ to $\rep(H)$. In terms of comodules, this functor can be described as follows. Let $\rho\colon V\to V\otimes_k k\{H\}$ be the comodule structure corresponding to the represenation $V$ and let $$R_g^*\colon {\hs (k\{H\})}=k\{H\}\otimes_kk \to k\{H\},\quad a\otimes b \mapsto g(a)\cdot b.$$ Then the comodule structure corresponding to the representation $\hs V$ is
$$g(\rho)\colon {\hs}V\xrightarrow{\hs\rho} {\hs}V\otimes_k {\hs(k\{H\})}\xrightarrow{\id\otimes R_g^*}{\hs V}\otimes_k k\{H\}.$$

\section{Tannakian Categories with Semigroup Actions} \label{sec: Difference Tannakian categories}

Let $H$ be a $G$-algebraic group and $H^\sharp$ the group scheme obtained from $H$ by forgetting the difference structure. Then the category of representations of $H$ (as a $G$-algebraic group) is equivalent to the category of representations of $H^\sharp$ (as a group scheme). However, intuitively it is clear that the representation theory of $H$ (as a $G$-algebraic group) is much richer than the representation theory of $H^\sharp$ (as a group scheme). The main point of this section is to identify, in a rather formal manner, an additional ``difference structure'' on the category of representations of $G$ which accounts for this purported richness. One can recover $H$ (as a $G$-algebraic group) from its (Tannakian) category of representations and this additional difference structure.

The main result in this section (Theorem \ref{theo: sTannakianmain}) is a purely categorical characterization of those categories that are categories of representations of group $G$-schemes. This is an analogue of the Tannaka duality theorem for group schemes. In the general context of fields with operators, a Tannaka duality theorem was proven in \cite{Kamensky:TannakianFormalismOverFieldsWithOperators}. However, in the situation that we are considering here (the case of a semigroup action), it is possible to give a very simple definition of difference Tannakian categories and a rather direct proof of the corresponding Tannaka duality theorem. We have, therefore, chosen to include an independent self--contained proof of the Tannaka duality theorem for difference group schemes.

The use of Theorem~\ref{theo: sTannakianmain} in practice warrants an effective description of actions of a particular class of groups on categories. Lemma~\ref{lem:free} and Theorem~\ref{thm:main0} in Section~\ref{sec:semigroupaction} provide such a description for free products of free finitely generated abelian semigroups, which is the most popular class of semigroups that appears in the applications.

\subsection{Review of Tannakian categories}\label{sec:revtann}

We start by recalling the usual Tannakian formalism. Basic references for Tannakian categories are \cite{Saavedra:CategoriesTannakien,DeligneMilne:TannakianCategories,Deligne:categoriestannakien}. We mostly follow \cite{DeligneMilne:TannakianCategories} in the nomenclature:
\begin{itemize}
\item
A \emph{tensor category} is a category $\C$ together with
\begin{itemize}
\item
 a functor $\C\times\C\to\C$, $(X,Y)\rightsquigarrow X\otimes Y$ and 
 \item compatible associativity and commutativity constraints
$$X\otimes(Y\otimes Z)\simeq (X\otimes Y)\otimes Z, \quad X\otimes Y\simeq Y\otimes X$$
\end{itemize}
such that there exists an \emph{identity object} $(\1,e)$. The identity object is unique up to unique isomorphisms and induces a functorial isomorphism $X\simeq\1\otimes X$.
\item
If $\C$ is abelian and $\otimes$ is bi-additive, we speak of an \emph{abelian tensor category}.
In this case $R:=\End(\1)$ is a (commutative) ring, $\C$ is $R$-linear (via $X\simeq\1\otimes X$) and $\otimes$ is $R$-bilinear.
\item
Let $R$ be a ring. An \emph{abelian tensor category over $R$} is an abelian tensor category together with an isomorphism of rings $R\simeq\End(\1)$.
\item
Let $\C$ and $\D$ be tensor categories. A \emph{tensor functor} $\C\to \D$ is a pair $(F,\alpha)$ comprising a functor $F\colon \C\to \D$ and a functorial isomorphism
$\alpha_{X,Y}\colon F(X)\otimes F(Y)\simeq F(X\otimes Y)$ such that some natural properties are satisfied. If $\C$ and $\D$ are abelian, $F$ is required to be additive. We will often omit $\alpha$ from the notation and speak of $F$ as a tensor functor. A morphism of tensor functors is a morphism of functors also satisfying some natural properties.
\item
A tensor category is called \emph{rigid} if every object $X$ has a dual $X^\vee$ (cf. \cite[Definition~1.7]{DeligneMilne:TannakianCategories} and \cite[2.1.2]{Deligne:categoriestannakien}.)
\item
Let $k$ be a field. A \emph{neutral Tannakian category} over $k$ is a rigid abelian tensor category $\C$ over $k$, such that there exists an exact faithful $k$-linear tensor functor $\omega\colon\C\to\vect_k$. Any such functor is said to be a \emph{fibre functor} for $\C$.
\item
For every $k$-algebra $R$, composing $\omega$ with the canonical tensor functor $\vect_k\to\mod_R$, $V\rightsquigarrow V\otimes_k R$ yields a tensor functor $\omega\otimes R\colon\C\to\mod_R$. We can define a functor $\underline{\operatorname{Aut}}^\otimes(\omega)\colon \alg_k\to\groups$ by associating to every $k$-algebra $R$ the group of automorphisms of $\omega\otimes R$ (as tensor functor).
\end{itemize}

The main result about Tannakian categories is the following:
\begin{theo}[{\cite[Theorem 2.11]{DeligneMilne:TannakianCategories}}] \label{theo: Tannakianmain}
Let $\C$ be a neutral Tannakian category over $k$ and $\omega\colon\C\to\vect_k$ a fibre functor. Then $H=\underline{\operatorname{Aut}}^\otimes(\omega)$ is an affine group scheme over $k$ and $\omega$ induces an equivalence of tensor categories between $\C$ and the category of finite dimensional representations of $H$.
\end{theo}
For later use, we record a corollary.
\begin{cor} \label{cor: morphism of groups corresponds to functor}
Let $k$ be a field and $\C$, $\C'$ neutral Tannakian categories over $k$ with fibre functors $\omega$ and $\omega'$, respectively. There is a canonical bijection between the set of morphisms of group $k$-schemes from $H=\underline{\operatorname{Aut}}^\otimes(\omega)$ to $H'=\underline{\operatorname{Aut}}^\otimes(\omega')$ and the set of equivalence classes of pairs $(F,\alpha)$, where $F\colon\C'\to\C$ is a tensor functor and
$\alpha\colon\omega F\to \omega'$ an isomorphism of tensor functors. Another such pair $(F_1,\alpha_1)$ is equivalent to $(F,\alpha)$ if there exists an isomorphism of tensor functors $F\to F_1$ such that
\[\xymatrix{
 \omega F \ar[rr] \ar_{\alpha}[rd] & & \omega F_1 \ar^{\alpha_1}[ld] \\
  & \omega' &
}
\]
commutes.
\end{cor}
\begin{proof} This follows from Theorem~\ref{theo: Tannakianmain} and \cite[Corollary~2.9]{DeligneMilne:TannakianCategories}
\end{proof}

\subsection{Actions of semigroups on categories}\label{sec:semigroupaction}
Let $G$ be a semigroup. The main result of this section is Theorem~\ref{thm:main0}, which provides a finite set of diagrams defining actions of free finitely generated abelian semigroups on categories. By Lemma~\ref{lem:free}, this also implies that actions of finite free products of such groups on categories can be described using finite sets of diagrams. We conclude this section with Example~\ref{ex:commuteandnot}, in which we show that a commutative group can act on a differential equation in a non-commutative way, which is, in fact, commutative up to a gauge transformation. This justifies that the right notion of action in this context is on categories of differential modules rather than just on differential equations.

\subsubsection{Definition of action}
We will start with the main definition, which contains infinite data if and only if $G$ is infinite.

\begin{defi}[{see also \cite[Section~0]{Deligne1997}, \cite[Sections~1.3.3 and~1.3.4]{Frenkel}, and \cite[Section~4.1]{DGNO}}]\label{def:actionmain}
A \emph{$G$-category} is a category $\C$ together with a set of functors $$T(g) \colon \C\to\C,\quad g\in G,$$ and isomorphisms of functors
$$
c_{f,g} : T(f)\circ T(g) \to T(fg),\quad f,g \in G,\quad \quad \iota : T(e) \stackrel{\sim}{\longrightarrow} \id_\C,
$$
such that the following diagram is commutative:
\begin{equation}\label{diag:assoc}
\begin{CD}
T(f)\circ T(g)\circ T(h)@>c_{f,g}\circ\id>> T(fg)\circ T(h)\\
@V\id\circ c_{g,h}VV@V{c_{fg,h}}VV\\
T(f)\circ T(gh)@>c_{f,gh}>> T(fgh).
\end{CD}
\end{equation}
\end{defi}
$G$-actions on $\C$ form a category, denoted by $\C_G$, 
\begin{enumerate}
\item an object is a set of functors and isomorphisms $$\big(\{T(g)\:|\: g \in G\}, \big\{c_{f,g}\:|\: f,g\in G\big\}\big)$$ as above;
\item a morphism between two objects, $T$ and $T'$, is a set of morphisms of functors $$\big\{m_f : T(f)\to T'(f)\:|\: f \in G\big\}$$ such that the following diagram is commutative:
\begin{align}\label{diag:comm}
\begin{CD}
T(f)\circ T(g)@>c_{f,g}>> T(fg)\\
@Vm_f\circ m_gVV@Vm_{fg}VV\\
T'(f)\circ T'(g)@>c'_{f,g}>> T'(fg).
\end{CD}
\end{align}
\end{enumerate}

\begin{lemma}\label{rem:34}
Let $F : \C\to \C$ be a functor, $g \in G$, and $I : F \to T(g)$, an isomorphism of functors. Then $$\big(T,\big\{c_{f,g}\big\},\iota\big)\cong\left(T',\big\{c'_{f,g}\big\},\iota'\right),$$
where $$T'(h) := T(h),\  h \in G\setminus\{g\},\ \ T'(g) := F,\ \ c'_{f,h} = c_{f,h}, \  f,h\in G\setminus\{g\},\ \ c'_{g,h} :=c_{g,h}(I\circ\id),\ \ c'_{h,g}:= c_{h,g}(\id\circ I),\  h \in G,$$
and $\iota' = \iota$ if $g \ne e$ and $\iota':=\iota I$ if $g = e$.
 In particular, $$\big(T,\big\{c_{f,h}\big\},\iota\big) \cong \left(T',\big\{c'_{f,h}\big\},\id\right),$$ where, $$T'(h) := T(h),\  h \in G\setminus\{e\},\ \ T'(e) := \id_\C,\ \ c'_{f,h} = c_{f,h}, \  f,h\in G\setminus\{e\},\ \ c'_{e,h} :=c_{e,h}\big(\iota^{-1}\circ\id\big),\ \ c'_{h,e}:= c_{h,e}\big(\id\circ\iota^{-1}\big),\  h \in G.$$
 In other words, for a given element of the semigroup, one obtains an isomorphic action of the semigroup by replacing the action of this element by a functor that is isomorphic to it.
\end{lemma}
\begin{proof}Let $m_h = \id : T(h) \to T'(h)$, $h \in G \setminus \{g\}$ and $m_g = I^{-1} : T(g) \to F$. Then~\eqref{diag:assoc} and~\eqref{diag:comm} are commutative by the construction. 
\end{proof}

Now, a natural question is: for what classes of semigroups, their actions on categories can be defined using only finitely many data. For instance, can one find a {\it restriction functor} $R$ from $\C_G$ to the category of actions of a particular finite subset of $G$ (or some other finite subset of some other semigroup associated to $G$, as done in \cite[Th\'eor\`eme~1.5]{Deligne1997}) so that $R$ is an equivalence of categories? In the following two sections, we will show that this is the case for finite free products of semigroups of the form  $\N^n\times \Zb/{n_1}\Zb\times\ldots\times\Zb/{n_r}\Zb$.
\subsubsection{Actions of free products of semigroups}
In this section, we will show how to describe actions of free products of semigroups in terms of actions of each of the semigroups.

For every pair of semigroups $G_1$ and $G_2$, we have the category $\C_{G_1}\times\C_{G_2}$  \cite[\S II.3]{Maclane}.
We will define the restriction functor $$R : \C_{G_1*G_2} \to \C_{G_1}\times\C_{G_2}$$
as follows: 
\begin{enumerate}
\item for an object $$
T = \left(\left\{T(g):\C\to\C\:\big|\: g\in G_1*G_2\right\}, \left\{c_{f,g}\:\big|\: f, g \in G_1*G_2\right\}\right),$$
we let
$$
R(T) := \left(\left(\left\{T(g):\C\to\C\:\big|\: g\in G_1\right\}, \left\{c_{f,g}\:\big|\: f, g \in G_1\right\}\right),\left(\left\{T(g):\C\to\C\:\big|\: g\in G_2\right\}, \left\{c_{f,g}\:\big|\: f, g \in G_2\right\}\right)\right);
$$
\item for objects $T_1$ and $T_2$ and a morphism 
$$m = \left\{m_f :T_1(f)\to T_2(f), f \in G_1*G_2\right\},$$
we let
$$
R(m) := \left(\left\{m_f :T_1(f)\to T_2(f), f \in G_1\right\},\left\{m_f :T_1(f)\to T_2(f), f \in G_2\right\}\right).
$$
\end{enumerate}
\begin{lemma}\label{lem:free} For all semigroups $G_1$ and $G_2$, the restriction functor $R : \C_{G_1*G_2} \to \C_{G_1}\times\C_{G_2}$ is an equivalence of categories.
\end{lemma}
\begin{proof} We will show this by constructing a quasi-inverse functor $E$ to $R$. For every object $$\left(\left(\left\{T(g):\C\to\C\:\big|\: g\in G_1\right\}, \left\{c_{f,g}\:\big|\: f, g \in G_1\right\}\right),\left(\left\{T(g):\C\to\C\:\big|\: g\in G_2\right\}, \left\{c_{f,g}\:\big|\: f, g \in G_2\right\}\right)\right),$$ define
$$
E(T) = \left(\left\{T(u_1)\circ T(v_1)\circ\ldots\circ T(u_q)\circ T(v_q)\:\big|\: u_i \in G_1, v_i\in G_2, 1\le i\le q, q \ge 1\right\},\left\{c_{f,g}\:\big|\: f,g \in G_1*G_2\right\}\right),
$$
where, for all presentations of the shortest length $$f=u_1v_1\cdot\ldots\cdot u_rv_r,\quad g = u_1'v_1'\cdot\ldots\cdot u_s'v_s',\quad u_i,u_j' \in G_1,\ v_i,v_j' \in G_2,\ 1\le i\le r,\ 1\le j\le s,$$ define
\begin{equation}\label{eq:case1}
c_{f,g} = \id : T(f)\circ T(g) \to T(fg)
\end{equation}
if $v_r,u_1' \ne e$ or $v_ru_1' = e$. Otherwise, if $v_r = e$, define
\begin{equation}\label{eq:case2}
c_{f,g} = \id_{T(u_1)\circ T(v_1)\circ\ldots\circ T(u_{r-1})\circ T(v_{r-1})}\circ c_{u_r,u_1'}\circ\id_{T(v_1')\circ\ldots\circ T(u_s')\circ T(v_s')}.
\end{equation}
The case $u_1' = e$ is similar. The required associativity  for $c_{.,.}$follows from~\eqref{eq:case1} and~\eqref{eq:case2} and the associativity for $c_{.,.}$ in each of $G_1$ and $G_2$.

Let now $m \in \MorC_{\C_{G_1}\times\C_{G_2}}(T_1,T_2)$ and $f = u_1v_1\cdot\ldots\cdot u_rv_r$ with $u_i \in G_1$, $v_i\in G_2$, $1\le i\le r$, being a presentation of the shortest length. Define
$$
m_f := m_{u_1}\circ m_{v_1}\circ\ldots\circ m_{u_r}\circ m_{v_r}, 
$$
which satisfies~\eqref{diag:comm} by construction.
By construction as well, $$R\circ E = \id_{\C_{G_1}\times\C_{G_2}}.$$
We will show that $$E\circ R \cong \id_{\C_{G_1*G_2}}.$$
Indeed, let $T_1,T_2 \in \Ob\big(\C_{G_1*G_2}\big)$ and $m \in \MorC_{\C_{G_1*G_2}}(T_1,T_2)$. Then the diagram
$$
\begin{CD}
(E\circ R)(T_1)@>I_{T_1}>> T_1\\
@V(E\circ R)(m)VV@VmVV\\
(E\circ R)(T_2)@>I_{T_2}>> T_2
\end{CD}
$$
is commutative, where, for each $T \in \Ob\big(\C_{G_1*G_2}\big)$, the set of isomorphisms of functors $$I_T : (E\circ R)(T)\to T$$ is defined by, for each $u_1v_1\cdot\ldots\cdot u_rv_r \in G_1*G_2$, successively composing isomorphisms of functors of the form $$c_{u_1v_1\cdot\ldots\cdot u_iv_i,u_{i+1}v_{i+1}\cdot\ldots\cdot u_rv_r}\quad\text{and}\quad c_{u_1v_1\cdot\ldots\cdot u_i,v_iu_{i+1}v_{i+1}\cdot\ldots\cdot u_rv_r},$$
which finishes the proof.
\end{proof}

\subsubsection{Actions of finitely generated abelian semigroups}
In this section, we will discuss actions of finitely generated abelian semigroups of a special form: \begin{equation}\label{eq:G}
G =  \N^n\times \Zb/{n_1}\Zb\times\ldots\times\Zb/{n_r}\Zb,\quad n_j \geq 1,\ 1\leq j\leq r,
\end{equation} (for simplicity, let $m = n+r$) with a selected set  $A =\{a_1,\ldots,a_m\}$ of generators that correspond to the decomposition~\eqref{eq:G}. Let the category $\C_A$ consist of
\begin{enumerate}
\item objects of the form $$\hspace{-0.25in}\left(\left\{T(a)\:|\: a \in A\right\}, \left\{i_{a_i,a_j} : T(a_i)\circ T(a_j) \stackrel{\sim}{\longrightarrow} T(a_j)\circ T(a_i)\:\big|\:  1\le j< i \le m\right\}, \left\{I_j : T(a_{n+j})^{\circ n_j}\stackrel{\sim}{\longrightarrow} \id\:|\: 1\leq j\leq r \right\}\right)$$
such that, for all $i_1, i_2,i_3$, $1 \le i_1 < i_2 < i_3 \le m$, the following diagram  is commutative (the hexagon axiom):
\begin{equation}\label{eq:i1i2i3}
\xy
(0,0)*+{T(a_{i_3})\circ T(a_{i_2})\circ T(a_{i_1})}="a", (20,18)*+{T(a_{i_2})\circ T(a_{i_3}) \circ T(a_{i_1})}="b",
(70,18)*+{T(a_{i_2})\circ T(a_{i_1}) \circ T(a_{i_3})}="c", (90,0)*+{T(a_{i_1})\circ T(a_{i_2}) \circ T(a_{i_3})}="d",(70,-18)*+{T(a_{i_1})\circ T(a_{i_3}) \circ T(a_{i_2})}="e", (20,-18)*+{T(a_{i_3})\circ T(a_{i_1})\circ T(a_{i_2})}="f"
\ar^{i_{a_{i_3},a_{i_2}}\circ \id} @{->} "a";"b"
\ar^{\id\circ i_{a_{i_3},a_{i_1}}} @{->} "b";"c"
\ar^{i_{a_{i_2},a_{i_1}}\circ \id} @{->} "c";"d"
\ar_{ \id\circ i_{a_{i_2},a_{i_1}}} @{->} "a";"f"
\ar^{^{i_{a_{i_3},a_{i_1}}\circ \id}} @{->} "f";"e"
\ar_{\id\circ i_{a_{i_3},a_{i_2}}} @{->} "e";"d"
\endxy
\end{equation}
as well as, for all $j$, $1\le j \le r$,
\begin{equation}\label{eq:finordercond}
\begin{CD}
T(a_{n+j})\circ T(a_{n+j})^{\circ n_j-1}\circ T(a_{n+j}) @>>> T(a_{n+j})^{\circ n_j}\circ T(a_{n+j})\\
@VVV@VV I_j\circ\id V\\
 T(a_{n+j})\circ T(a_{n+j})^{\circ n_j}@>\id\circ I_j>> T(a_{n+j}).
\end{CD}
\end{equation}
\item morphisms between two objects $T$ and $T'$ consist of morphisms of functors 
$$\big\{m_a : T(a)\to T'(a) \:\big|\: a \in A\big\}$$ such that, for all $i,j$, $i < j$, $1\le i,j\le m$,
the following diagram is commutative:
\begin{equation}\label{diag:comm2}
\begin{CD}
T(a_i)\circ T(a_j) @>i_{a_i,a_j}>>T(a_j)\circ T(a_i)\\
@Vm_{a_i}\circ m_{a_j}VV@Vm_{a_j}\circ m_{a_i}VV\\
T'(a_i)\circ T'(a_j) @>i'_{a_i,a_j}>>T'(a_j)\circ T'(a_i),
\end{CD}
\end{equation}
and, for all $j$,  $1\le j \le r$, the following diagram is commutative:
\begin{equation}\label{diag:commfin}
\begin{CD}
T(a_{n+j})^{\circ n_j} @>I_j>>\id_\C\\
@Vm_{a_{n+j}}^{\circ n_j}VV@|\\
T'(a_{n+j})^{\circ n_j}  @>I_j'>>\id_\C.
\end{CD}
\end{equation}
\end{enumerate}
The restriction functor $R : \C_G \to \C_A$ is defined as follows:
\begin{align}
 &R\big(\{T(g)\:|\: g \in G\}, \big\{c_{f,g}\:\big|\: f,g\in G\big\}\big) \notag\\
 &= \left(\{T(a)\:|\: a \in A\}, \left\{ i_{a_i,a_j} :=c_{a_j,a_i}^{-1} \circ c_{a_i,a_j},\ i> j\right\},\left\{I_j:= \iota\circ c_{a_{n+j},a_{n+j}^{n_j-1}}\circ\ldots\circ c_{a_{n+j},a_{n+j}}, 1\leq j\leq r\right\}\right),\label{eq:R1}\\ 
 &R\left(\left\{m_g\:\big|\: g\in G\right\}\right) = \big\{m_a\:\big|\: a \in A\big\},\label{eq:R2}
 \end{align}
where one shows that the latter satisfies~\eqref{diag:comm2} and~\eqref{diag:commfin} by combining several diagrams~\eqref{diag:comm}, for $c_{a_i,a_j}$, $c_{a_j,a_i}$, and $c_{a_i,a_i}$. Moreover,~\eqref{eq:i1i2i3} is satisfied. Indeed, we denote: $$T_i := T(a_i),\quad T_{ij} := T(a_ia_j),\quad T_{ijk} := T(a_ia_ja_k),\quad  i,j,k=1,2,3,$$ and, for simplicity, omit the composition sign. Note that $T_{ij} = T_{ji}$ and $T_{ijk}=\ldots=T_{kji}$. We have:
$$
\begin{CD}
T_3T_2T_1@>\rd{c_{a_3,a_2}\circ\id}>>T_{32}T_1@>\rd{c_{a_2,a_3}^{-1}\circ\id}>>T_2T_3T_1@.T_2T_{13}@>c_{a_2,a_1a_3}>>T_{123}\\
@VV\rd{\id\circ c_{a_2,a_1}}V@VVV@VV\rd{\id\circ c_{a_3,a_1}}V@AA\id\circ c_{a_1,a_3}A@AAc_{a_1a_2,a_3}A\\
T_3T_{21}@>c_{a_3,a_1a_2}>>T_{321}@<c_{a_2,a_1a_3}<<T_2T_{31}@>\rd{\id\circ c_{a_1,a_3}^{-1}}>>T_2T_1T_3@>\rd{c_{a_2,a_1}\circ\id}>>T_{21}T_3\\
@VV\rd{\id\circ c_{a_1,a_2}^{-1}}V@AAc_{a_1a_3,a_2}A@.@.@VV\rd{c_{a_1,a_2}^{-1}\circ\id} V\\
T_3T_1T_2@>\rd{c_{a_3,a_1}\circ\id}>>T_{31}T_2@>\rd{c_{a_1,a_3}^{-1}\circ\id}>>T_1T_3T_2@>\rd{\id\circ c_{a_3,a_2}}>>T_1T_{32}@>\rd{\id\circ c_{a_2,a_3}^{-1}}>>T_1T_2T_3\\
@.@.@VVc_{a_1,a_3}\circ\id V@VVc_{a_1,a_2a_3}V@VVc_{a_1,a_2}\circ\id V\\
@.@.T_{13}T_2@>c_{a_1a_3,a_2}>>T_{123}@<c_{a_1a_2,a_3}<<T_{21}T_3
\end{CD}
$$
We then have
\begin{align*}
c_{a_2,a_1}\circ\id&=c_{a_1a2,a_3}^{-1}\circ c_{a_2,a_1a_3}\circ(\id\circ c_{a_1,a_3}),\\
 \left(\id\circ c_{a_2,a_3}^{-1}\right)(\id\circ c_{a_3,a_2})&=\left(c_{a_1,a_2}^{-1}\circ\id\right)\circ c_{a_1a_2,a_3}^{-1}\circ c_{a_1a3,a_2}\circ (c_{a_1,a_3}\circ\id).
\end{align*}
Therefore,
$$
 \left(\id\circ c_{a_2,a_3}^{-1}\right)(\id\circ c_{a_3,a_2})=\left(c_{a_1,a_2}^{-1}\circ\id\right)\circ(c_{a_2,a_1}\circ\id)\left(\id\circ c_{a_1,a_3}^{-1}\right)\circ c_{a_2,a_1a_3}^{-1}\circ c_{a_1a3,a_2}\circ (c_{a_1,a_3}\circ\id),
$$
which finally shows the required equality of the two paths of isomorphisms of functors highlighted in blue starting at $T_3T_2T_1$ and ending at $T_1T_2T_3$.

Finally,~\eqref{eq:finordercond} is a direct consequence of intreated applications of~\eqref{diag:assoc}.
\begin{theo}\label{thm:main0} The functor of restriction $R : \C_G \to \C_A$ defined above is an equivalence of categories.
\end{theo}
\begin{proof} 
First note that, by Lemma~\ref{rem:34}, we may assume that $\iota = \id_\C$.
We will construct a quasi-inverse functor $E$ to $R$. For this, let $T \in \Ob(\C_A)$. We define 
$$
E(T) = \left(\left\{T\left(a_1^{d_1}\cdot\ldots\cdot a_m^{d_m}\right)\:\Big|\:{d_i \ge 0, 1\le i\le m,\atop\ d_i < n_i,\ n<i\le m}\ \right\},\left\{c_{a_1^{s_1}\cdot\ldots\cdot a_m^{s_m},a_1^{q_1}\cdot\ldots\cdot a_m^{q_m}}\:\Big|\: {{s_i,q_i \ge 0, 1\le i \le m}\atop{s_i, q_i < n_i, n<i\le m}}\right\}\right),
$$
where
\begin{equation}\label{eq:Tai}
T\left(a_1^{d_1}\cdot\ldots\cdot a_m^{d_m}\right) := T(a_1)^{\circ d_1}\circ\ldots\circ T(a_m)^{\circ d_m},\quad T(e) := \id_\C,
\end{equation}
$d_{n+j}$ is taken modulo $n_j$, $1\leq j\leq r$,
and each
$c_{a_1^{s_1}\cdot\ldots\cdot a_m^{s_m},a_1^{q_1}\cdot\ldots\cdot a_m^{q_m}}$ is defined as the appropriate composition of isomorphisms of functors 
\begin{equation}\label{eq:caij}
\id_{T(a_p)^{\circ d}},\ \ i_{a_i,a_j},\ \ I_s,\quad 1\le i,j,p \le m,\ i>j,\ d >0,\ 1\le s\le r,
\end{equation}
that corresponds to turning $a_1^{s_1}\cdot\ldots\cdot a_m^{s_m}\cdot a_1^{q_1}\cdot\ldots\cdot a_m^{q_m}$ into $a_1^{s_1+q_1}\cdot\ldots\cdot a_m^{s_m+q_m}$ by successively exchanging the adjacent powers of $a_i$'s in this product starting with moving $a_m^{s_m}$ to the position next to $a_m^{q_m}$, then similarly continuing with $a_{m-1}^{s_{m-1}}$, and so on, and computing modulo the $n_j$'s whenever needed. Note that we have fixed the above particular way of the successive exchanges, and it will be used later. Finally, $$R(\{m_a\:|\: a \in A\}) := \{m_g\:|\: g\in G\},$$ where each $m_g$ is defined as the appropriate composition of the $m_a$'s following~\eqref{eq:Tai}.

To show that the associativity condition~\eqref{diag:assoc} holds, first note that~\eqref{eq:i1i2i3} implies~\eqref{diag:assoc} for all triples 
\begin{equation}\label{eq:triples}
\left(a_{i_1},a_{i_2},a_{i_3}\right),\quad 1\le i_1,i_2,i_3\le m.
\end{equation}
 Indeed, if $i_1=i_2=i_3 > n$ and  $n_{i_1} =2$, then~\eqref{diag:assoc} follows from~\eqref{eq:finordercond}.
 Now, if $i < j$, then $c_{a_i,a_j} = \id$, so it is sufficient to deal with the triples~\eqref{eq:triples} with $i_1 > i_2 >  i_3$, which is done in~\eqref{eq:i1i2i3}, taking into account that, by~\eqref{eq:Tai} and~\eqref{eq:caij},
\begin{align*}
T(a_{i_2}) T(a_{i_3}) T(a_{i_1}) &= T(a_{i_2}a_{i_3}) T(a_{i_1})=T(a_{i_3}a_{i_2}) T(a_{i_1}),\\
T(a_{i_3}) T(a_{i_1}) T(a_{i_2})&=T(a_{i_3}) T(a_{i_1}a_{i_2})=T(a_{i_3})T(a_{i_2}a_{i_1}),\\
T(a_{i_1}) T(a_{i_2}) T(a_{i_3}) &= T(a_{i_1}a_{i_2}a_{i_3})=T(a_{i_3}a_{i_2}a_{i_1})
\end{align*} 
and
\begin{align*}
&c_{a_{i_3},a_{i_2}} = i_{a_{i_3},a_{i_2}}, &c_{a_{i_3}a_{i_2},a_{i_1}}=\big(i_{a_{i_2},a_{i_1}}\circ \id\big)\circ\big(\id\circ i_{a_{i_3},a_{i_1}}\big),\\ 
&c_{a_{i_2},a_{i_1}}=i_{a_{i_2},a_{i_1}}, &c_{a_{i_3},a_{i_2}a_{i_1}}=\big(\id\circ i_{a_{i_3},a_{i_2}}\big)\circ \big(i_{a_{i_3},a_{i_1}}\circ \id\big).
\end{align*}
The general case will be shown by induction on the triples $(f,g,h)$ ordered degree-lexicographically, where each entry of the triple is ordered degree-lexicographically as well (that is, the degree-lexicographical order on $\N^{3m}$).  The base case, in which $$(f,g,h)=\left(a_{i_1},a_{i_2},a_{i_3}\right),$$ has been done above. 
Moreover, note that, if $f=e$, or $g = e$, or $h=e$, then the statement is a tautology.
Let us show~\eqref{diag:assoc} for a triple $(f,g,h)$ with $f \ne e$ . Let $$f = a_if',$$ where $f'$ does not have $a_1,\ldots,a_{i-1}$ in it. Then, by~\eqref{eq:Tai},
$$
T(a_i) T\left(f'\right) = T\left(a_if'\right)
$$
and, therefore,
$$
T(f) T(g) T(h) = T(a_i)T\left(f'\right) T(g)T(h).
$$
If $f' \ne e$, then, by the inductive hypothesis, all squares in the diagram below are commutative 
$$
\begin{CD}
T(a_i)T\left(f'\right) T(g) T(h)@>\id\circ c_{f',g}\circ\id>>T(a_i) T\left(f'g\right) T(h)@>c_{a_i,f'g}\circ\id>>T\left(a_if'g\right) T(h)\\
@VV\id\circ\id\circ c_{g,h}V@VV\id\circ c_{f'g,h}V@Vc_{a_if'g,h}=c_{fg,h}VV\\
T(a_i) T\left(f'\right) T(gh)@>\id\circ c_{f',gh}>>T(a_i) T\left(f'gh\right)@>c_{a_i,f'gh}>>T\left(a_if'gh\right)=T(fgh)\\
@VVc_{a_i,f'}\circ\id = \id\circ\id V@VVc_{a_i,f'gh}V\\
T\left(a_if'\right)T(gh)=T(f) T(gh)@>c_{a_if',gh}=c_{f,gh}>>T\left(a_if'gh\right)=T(fgh)
\end{CD}
$$
and, by the diagram for the triple $\left(a_i,f',g\right)$, $$c_{a_i,f'g}\circ \left(\id\circ c_{f',g}\right) = c_{a_if',g}\circ \left(c_{a_i,f'}\circ\id\right) = c_{a_if',g}\circ (\id\circ\id) = c_{f,g},$$ which shows~\eqref{diag:assoc}  for $(f,g,h)$ in the case $f \ne a_i$. 

We now continue with the case of triples of the form $(a_i,g,h)$ as above by representing $g = a_jg'$ if $g \ne e$.  So, $c_{a_j,g'} = \id$.
  If $g'\ne e$ and $i>j$, then $c_{a_j,a_i} = \id$, and the commutativity (by the inductive hypothesis) of the square diagrams below:
$$ 
\begin{CD}
T(a_i)T(g)T(h)=T(a_i)T(a_j)T\big(g'\big)T(h)@>\id\circ c_{a_j,g'}\circ\id>>T(a_i)T\big(a_jg'\big)T(h)=T(a_i)T(g)T(h)\\
@VVc_{a_i,a_j}\circ\id\circ\id V@VVc_{a_i,g}\circ\id V\\
T(a_j)T(a_i)T\big(g'\big)T(h)@>\id\circ c_{a_i,g'}\circ\id>>T\big(a_ia_jg'\big)T(h)=T(a_ig)T(h)=T(a_j)T\big(a_ig'\big)T(h)\\
@VV\id\circ\id\circ c_{g',h}V@VV \id\circ c_{a_ig',h} V\\
T(a_j)T(a_i)T\big(g'h\big)@>\id\circ c_{a_i,g'h}>>T(a_j)T\big(a_ig'h\big)\\
@VV c_{a_j,a_i}\circ\id=\id\circ\id V@VVc_{a_j,a_ig'h}V\\
T(a_ja_i)T\big(g'h\big)@>c_{a_ja_i,g'h}>>T\big(a_ja_ig'h\big)=T(a_igh)
\end{CD}
$$
and
$$
\begin{CD}
T(a_j)T\big(a_ig'\big)T(h)@>\id\circ c_{a_ig',h}>>T(a_j)T\big(a_ig'h\big)\\
@VV c_{a_j,a_ig'}\circ\id=\id\circ\id\circ\id V@VVc_{a_j,a_ig'h}V\\
T\big(a_ja_ig'\big)T(h)@>c_{a_ig,h}>>T\big(a_ja_ig'h\big)=T(a_igh)
\end{CD}\quad
\begin{CD}
T(a_i)T(a_j)T\big(g'h\big)@>c_{a_i, a_j}\circ\id>>T(a_ia_j)T\big(g'h\big)\\
@VV \id\circ c_{a_j,g'h} V@VVc_{a_ia_j,g'h}V\\
T(a_i)T\big(a_jg'h\big)@>c_{a_i,gh}>>T(a_igh),
\end{CD}
$$
shows that
\begin{align*}
c_{a_ig,h}\big(c_{a_i,g}\circ\id\big)&=c_{a_j,a_ig'h}\big(\id\circ c_{a_ig',h}\big)\big(c_{a_i,g}\circ\id\big)=c_{a_ja_i,g'h}\big(c_{a_i,a_j}\circ c_{g',h}\big)\\
&= c_{a_i,gh}\big(\id\circ c_{a_j,g'h}\big)\big(\id\circ\id\circ c_{g',h}\big)
=c_{a_i,gh}\big(\id\circ c_{g,h}\big)\big(\id\circ c_{a_j,g'}\circ\id\big)=c_{a_i,gh}\big(\id\circ c_{g,h}\big).
\end{align*}
This implies~\eqref{diag:assoc} in this case as well. 
If $g'\ne e$ and $i \leq j$, then the following square diagrams are commutative by the inductive hypothesis:
$$ 
\begin{CD}
T(a_i)T(a_j)T\big(g'\big)T(h)@>\id\circ c_{a_j,g'}\circ\id>>T\big(a_ia_jg'\big)T(h)=T(a_ig)T(h)=T(a_i)T\big(a_jg'\big)T(h)\\
@VV\id\circ\id\circ c_{g',h}V@VV \id\circ c_{a_jg',h}=\id\circ c_{g,h} V\\
T(a_i)T(a_j)T\big(g'h\big)@>\id\circ c_{a_j,g'h}>>T(a_i)T\big(a_jg'h\big)\\
@VV c_{a_i,a_j}\circ\id=\id\circ\id V@VVc_{a_i,a_jg'h}=c_{a_i,gh}V\\
T(a_ia_j)T\big(g'h\big)@>c_{a_ia_j,g'h}>>T\big(a_ia_jg'h\big)=T(a_igh)
\end{CD}
$$
and
$$
\begin{CD}
T(a_i)T(g)T(h)=T(a_i)T(a_j)T\big(g'\big)T(h)@>\id\circ c_{g',h}>>T(a_ia_j)T\big(g'h\big)\\
@VVc_{a_i,g}\circ\id=\id\circ\id V@VVc_{a_ia_j,g'h}V\\
T(a_ig)T(h)@>c_{a_ig,h}>>T(a_igh),
\end{CD}
$$
with the latter following directly from the definition of the isomorphisms $c$~\eqref{eq:caij} and from~\eqref{eq:finordercond} if $i >n$, $j=i$, $g = a_i^{n_i-1}g''$, and $h$ contains $a_i$.
This implies~\eqref{diag:assoc} in this case too.

Therefore, it remains to treat the case of a triple of the form $(a_i,a_j,h)$. 
As before, let $h=a_lh'$, with $h'$ containing no $a_l'$ with $l' < l$, and suppose that $h'\ne e$. If $l \ge \max\{i,j\}$, then~\eqref{diag:assoc} holds by the definition of the isomorphsims $c$~\eqref{eq:caij} and additionally, if $i=j=l>n$ and $n_i=2$, it follows from~\eqref{eq:finordercond}. If $l < j < i$, then, by definition and the inductive hypothesis, we have the commutativity of the diagrams
$$
\begin{CD}
T(a_i)T(a_j)T(a_l)T\big(h'\big)@>\id\circ c_{a_j,a_l}\circ\id>>T(a_i)T(a_la_j)T\big(h'\big)@>c_{a_i,a_l}\circ\id>>T(a_l)T(a_i)T(a_j)T\big(h'\big)\\
@VVc_{a_i,a_j}\circ\id V@.@VV\id\circ c_{a_i,a_j}\circ\id V\\
T(a_j)T(a_i)T(a_l)T\big(h'\big)@>\id\circ c_{a_i,a_l}\circ\id>>T(a_j)T(a_la_i)T\big(h'\big)@>c_{a_j,a_l}\circ\id=c_{a_j,a_la_i}\circ\id>>T(a_la_ja_i)T\big(h'\big)\\
@VV\id V@VV\id\circ c_{a_la_i,h'}V@VV\id\circ c_{a_ja_i,h'}V\\
T(a_j)T(a_i)T(h)@>\id\circ c_{a_i,h}>>T(a_j)T(a_ih)@>c_{a_j,a_ih}>>T(a_ia_jh)
\end{CD}
$$
and
$$
\begin{CD}
T(a_j)T(a_i)T(h)@>\id\circ c_{a_i,h}>>T(a_j)T(a_ih)\\
@VV\id V@VVc_{a_j,a_ih}V\\
T(a_ja_i)T(h)@>c_{a_ia_j,h}>>T(a_ia_jh)
\end{CD}\quad\quad\quad
\begin{CD}
T(a_l)T(a_i)T(a_j)T\big(h'\big)@>\id\circ c_{a_j,h'}>> T(a_l)T(a_i)T\big(a_jh'\big)\\
@VV\id\circ c_{a_i,a_j}\circ\id V@VV\id\circ c_{a_i,a_jh'}V\\
T(a_l)T(a_ja_i)T\big(h'\big)@>\id\circ c_{a_ja_i,h'}>>T(a_ia_jh)
\end{CD},
$$
which show that
\begin{align*}
c_{a_ia_j,h}\big(c_{a_i,a_j}\circ\id\big)&=c_{a_j,a_ih}\big(\id\circ c_{a_i,h}\big)\big(c_{a_i,a_j}\circ\id\big)\\
&=\big(\id\circ c_{a_ja_i,h'}\big)\big(\id\circ c_{a_i,a_j}\circ\id\big)\big(c_{a_i,a_l}\circ\id\big)\big(\id\circ c_{a_j,a_l}\circ\id\big)\\
&=\big(\id\circ c_{a_i,a_jh'}\big)\big(\id\circ c_{a_j,h'}\big)\big(c_{a_i,a_l}\circ\id\big)\big(\id\circ c_{a_j,a_l}\circ\id\big)\\
&=\big(\id\circ c_{a_i,a_jh'}\big)\big(c_{a_i,a_l}\circ\id\big)\big(\id\circ c_{a_j,h'}\big)\big(\id\circ c_{a_j,a_l}\circ\id\big)=c_{a_i,a_jh}\big(\id\circ c_{a_j,h}\big).
\end{align*}
Hence, we have shown~\eqref{diag:assoc} in this case.
If $j < l < i$, then, by definition, the following diagram is commutative
$$
\begin{CD}
T(a_i)(a_j)T(a_l)T\big(h'\big) @>c_{a_i,a_jh}>> T(a_ia_jh)\\
@VVc_{a_i,a_j}\circ\id V@AA \id\circ c_{a_i,h'}A\\
T(a_ia_j)T(a_l)T\big(h'\big)@>\id\circ c_{a_i,a_l}\circ\id>>T(a_ja_la_i)T\big(h'\big),
\end{CD}
$$
$c_{a_j,h} = \id$, and $$c_{a_ja_i,h} = \id\circ c_{a_i,h}=\id\circ c_{a_i,a_l}\circ\id=\id\circ c_{a_i,h'}.$$ Hence, we have shown~\eqref{diag:assoc} in this case as well. Finally, if $i < j$, then, by definition,
$$
\begin{CD}
T(a_i)T(a_j)T(h)@>\id\circ c_{a_j,h}>> T(a_i)T(a_jh)\\
@VV\id=c_{a_i,a_j} V@VV c_{a_i,a_jh} V\\
T(a_ia_j)T(h)@>c_{a_ia_j,h}>>T(a_ia_jh)
\end{CD}
$$
is commutative, showing~\eqref{diag:assoc} in this case too.
Therefore, we have reduced showing~\eqref{diag:assoc} to the case of a triple $(a_i,a_j,a_l)$, which is the base of the induction. This completes our induction, showing~\eqref{diag:assoc} for all triples $(f,g,h)$.

For all $T_1,T_2 \in \Ob(\C_A)$ and $m \in \MorC_{\C_A}(T_1,T_2)$, we define $$E(m) := \left\{m_g\:\big|\: g \in G\right\},$$ where
$$
m_{a_1^{d_1}\cdot\ldots\cdot a_m^{d_m}} := m_{a_1}^{\circ d_1}\circ\ldots\circ m_{a_m}^{\circ d_m}.
$$
By~\eqref{diag:comm2} and~\eqref{eq:caij}, for all $f,g \in G$, diagram~\eqref{diag:comm} is commutative.

By construction~\eqref{eq:R1} and~\eqref{eq:R2}, $R\circ E = \id_{\C_A}$. 
It remains to show that $$E\circ R \cong \id_{\C_G}.$$
Indeed, let $T_1,T_2 \in \Ob\big(\C_G\big)$, $m \in \MorC_{\C_G}(T_1,T_2)$, and $g \in G$. Then the diagram of morphisms of functors
$$
\begin{CD}
(E\circ R)(T_1)(g)@>I_{T_1,g}>> T_1(g)\\
@V(E\circ R)(m)_gVV@Vm_gVV\\
(E\circ R)(T_2)(g)@>I_{T_2,g}>> T_2(g)
\end{CD}
$$
is commutative, where, for each $T \in \Ob\big(\C_G\big)$, the isomorphism of functors $$I_{T,g} : (E\circ R)(T)(g)\to T(g)$$ is defined by successively composing isomorphisms of functors of the form $$c_{a_1^{d_1}\cdot\ldots\cdot a_i^{d_i},a_{i+1}^{d_{i+1}}\cdot\ldots\cdot a_m^{d_m}},$$
which completes the proof.
\end{proof} 

\begin{ex}\label{ex:counter} If one does not require that~\eqref{eq:i1i2i3} hold, one can obtain a non-associative semigroup action. Indeed, let $\C = \vect_\Qb$, $A = \{a_1,a_2,a_3\}$. Define $$T(a_i)(V)= \Qb\oplus V,\quad T(a_i)(\varphi) = \id_\Qb\oplus\varphi,\quad V,W \in \Ob(\C),\ \varphi \in \Hom(V,W),\ i=1,2,3.$$ For all $M \in \GL_2(\Qb)$ and $1\le i,j\le 3$, $\phi_M\oplus\id$ defines an isomorphism $T(a_i)\circ T(a_j) \to T(a_j)\circ T(a_i)$, where $\phi_M : \Qb^2\to \Qb^2$ is multiplication by $M$. Then, for all $M_1,M_2,M_3 \in \GL_2(\Qb)$ such that
$$
\begin{pmatrix}
M_1 &0\\
0& 1
\end{pmatrix}
\begin{pmatrix}
1 &0\\
0& M_2
\end{pmatrix}
\begin{pmatrix}
M_3 &0\\
0& 1
\end{pmatrix}
\ne
\begin{pmatrix}
1 &0\\
0& M_3
\end{pmatrix}
\begin{pmatrix}
M_2 &0\\
0& 1
\end{pmatrix}
\begin{pmatrix}
1 &0\\
0& M_1
\end{pmatrix},
$$
diagram~\eqref{eq:i1i2i3} is not commutative if we set $$i_{a_3,a_2} = \phi_{M_1}\oplus\id,\quad i_{a_3,a_1}=\phi_{M_2}\oplus\id,\quad\text{and}\quad i_{a_2,a_1} = \phi_{M_3}\oplus\id.$$ For instance, we can take
$$
M_1 = \begin{pmatrix}
1&1\\
0&-1
\end{pmatrix},\quad
M_2 = \begin{pmatrix}
1&1\\
0&-1
\end{pmatrix},\quad\text{and}\quad
M_3 = \begin{pmatrix}
-1&1\\
0&1
\end{pmatrix}.
$$
\end{ex}
\begin{ex} Let $\C=\vect_{\overline{\Qb}(t)}$, $n \ge 2$, and $a$ be a primitive $n$th root of unity. Then $$\sigma : \overline{\Qb}(t) \to \overline{\Qb}(t),\quad t\mapsto at,$$ defines a field automorphism of $\overline{\Qb}(t)$ of order $n$. Define an  action $T$ of $\Zb/n\Zb$ on $\C$ by
$$
T(1): V \mapsto {^\sigma\!}V := V\otimes_{\overline{\Qb}(t)}\overline{\Qb}(t),\quad fv\otimes 1 = v\otimes\sigma(f),\ \ v\in V, \ f\in \overline{\Qb}(t),
$$
as in Section~\ref{subsec:action}. Note that, for every $b\in \overline{\Qb(t)}$ and the isomorphisms
$$
I : T(1)^n(V) \to V,\ \ v\otimes 1\otimes\ldots\otimes 1\mapsto bv,\quad V \in \Ob(\C),\ \ v \in V,
$$ 
the diagram~\eqref{eq:finordercond} is commutative (and the action, therefore, satisfies~\eqref{diag:assoc}) if and only if $\sigma(b)=b$. For example, for $b=t$, $\sigma(b)\ne b$. This shows that, in general, one cannot avoid the requirement~\eqref{eq:finordercond}.
\end{ex}

\begin{ex}\label{ex:commuteandnot}
Let $K = \Cb(\exp(x),x)$, $\partial_x=\frac{d}{dx}$, $\sigma_1(x) =x+s_1$, $\sigma_2(x)=s_2x$, $s_1\in \Cb\setminus\{0\}$, $s_2\in\Zb\setminus\{0,1\}$. Note that 
\begin{equation}\label{eq:commutation}
\partial_x\sigma_1=\sigma_1\partial_x\quad\text{and}\quad \partial_x\sigma_2=s_2\sigma_2\partial_x.
\end{equation} Consider the differential equation:
\begin{equation}\label{eq:xy}
\partial_x(y)=xy.
\end{equation}
Every object in the rigid tensor category $\C$ generated by the differential module associated to~\eqref{eq:xy} and all its iterates under $\sigma_1$ and $\sigma_2$ (see \cite[Section~2.2]{SingerPut:differential} and Example~\ref{ex:diffmod}) is a direct sum of differential modules associated to differential equations of the form
\begin{equation}\label{eq:objectofC}
\partial_x(y) = (nx+m)y,\quad n\in\Zb,\ m \in\Cb.
\end{equation}
Applying $\sigma_1\circ \sigma_2$ to~\eqref{eq:objectofC}, we obtain, using~\eqref{eq:commutation},
\begin{equation}\label{eq:tx2}
\partial_x(y)=\big(s_2^2nx+s_2m+s_1s_2^2n\big){x}y,
\end{equation}
while, applying $\sigma_2\circ\sigma_1$ to~\eqref{eq:objectofC}, we arrive at, using~\eqref{eq:commutation},
\begin{equation}\label{eq:tx3}
\partial_x(y)=\big(s_2^2nx+s_2m+s_1s_2n\big)y.
\end{equation}
Note that~\eqref{eq:tx2} and~\eqref{eq:tx3} are gauge-equivalent over $K$ if and only if  the differential equation in $c$ 
$$
\big(s_2^2nx+s_2m+s_1s_2^2n\big)-\big(s_2^2nx+s_2m+s_1s_2n\big)= s_1s_2n(s_2-1) = \frac{\partial_x(c)}{c}
$$
has a solution $\exp\big((s_2-1)s_1s_2nx\big) \in K$ if and only if $(s_2-1)s_1s_2n \in \Zb$. This is the case uniformly for all $n \in \Zb$ if and only if 
\begin{equation}\label{eq:condition}
 s_1(s_2-1)s_2 \in \Zb.
\end{equation}
Assume that~\eqref{eq:condition} holds for the rest of the example.
Then, if we let $G = \N\times \N$ and define an action $T$ of $G$ on $\C$  by 
$$
G \ni (1,0) =:a_1 \mapsto \sigma_1,\quad G \ni (0,1)=:a_2 \mapsto \sigma_2,
$$
that is, for all objects $M$ and $N$ of $\C$,
$$
T(a_i)(M) := M\otimes_{\sigma_i} K,\quad \left(T(a_i)(\varphi)\right)(m\otimes 1) := \varphi(m)\otimes 1,\ \ \varphi\in\Hom(M,N),\ m\in M,\quad i=1,2,$$ 
then we obtain isomorphisms of functors
$$
i_{a_1,a_2} : T(a_1)\circ T(a_2) \to T(a_2)\circ T(a_1).
$$ 
Indeed, for all one-dimensional objects (the rest of the objects in $\C$ are just direct sums of these) $M$ and $N$ of $\C$ and all $\varphi \in \Hom(M,N)$, we have the commutative diagram
\begin{equation}\label{CD:iT}
\begin{CD}
M\otimes_{\sigma_1}K\otimes_{\sigma_2}K @>{i_{a_1,a_2}}_M>> M\otimes_{\sigma_2}K\otimes_{\sigma_1}K\\
@VVT(a_1)\left(T(a_2)(\varphi)\right)V @VVT(a_2)\left(T(a_1)(\varphi)\right)V\\
N\otimes_{\sigma_1}K\otimes_{\sigma_2}K @>{i_{a_1,a_2}}_N>> N\otimes_{\sigma_2}K\otimes_{\sigma_1}K
\end{CD}
\end{equation}
where $${i_{a_1,a_2}}_M = \exp(n(M)x)\cdot\id_M\otimes\id_K\otimes\id_K\quad\text{and}\quad {i_{a_1,a_2}}_N = \exp(n(N)x)\cdot\id_N\otimes\id_K\otimes\id_K.$$ 
Indeed, if $M = \Span\{e\}$ and $N = \Span\{f\}$, with $\partial_x(e)=ae$, $\partial_x(f)=bf$, and $\varphi(e)=df$, then
$$\partial_x(d)f+dbf=\partial_x(df)=\partial_x(\varphi(e))=\varphi(\partial_x(e))= \varphi(ae)=a\varphi(e)=adf.
$$
Therefore, \begin{equation}\label{eq:morphism}\partial_x(d)=(a-b)d
\end{equation}
and, if $n(M) \ne n(N)$, then~\eqref{eq:morphism}, as a differential equation in $d$, has no non-zero solutions in $K$, which proves the commutativity of~\eqref{CD:iT}.
This example shows that $i_{a_1,a_2}$ is not necessarily $\id$ for a commutative semigroup.

Note that, given~\eqref{eq:condition}, depending on the actual denominator of $s_1$, one of $$K\big(\exp(x^2/2)\big),\quad K\big(\exp(x^2/2),\exp(x/s_2)\big),\quad K\big(\exp(x^2/2),\exp(x/(s_2-1))\big),\quad K\big(\exp(x^2/2),\exp(x/(s_2(s_2-1))\big)$$ is a minimal $\big\{\partial_x,\sigma_1,\sigma_2\big\}$-field extension of $K$ that contains a non-zero solution of~\eqref{eq:xy} and whose $\partial_x$-constants coincide with $\Cb$ (cf. \cite[Definition~2.4]{DHWDependence}).
\end{ex}

We will now see how the classical contiguity relations for the hypergeometric functions are reflected in our Tannakian approach.
\begin{ex}\label{ex:hyper} Let $K = \Cb(a,b,c,z)$, and consider it as a $\N^3$-field, with the action of the generators $\sigma_1$, $\sigma_2$, and $\sigma_3$ defined as $$\sigma_1(a) = a+1,\quad \sigma_2(b)= b+1,\quad \sigma_3(c) = c+1.$$ Let $M$ be the differential module corresponding to the hypergeometric differential equation
\begin{equation}\label{eq:hg}
z(1-z)y''+\big(c-(a+b+1)z\big)y'-aby=0,
\end{equation}
whose companion matrix is
$$
A:=\begin{pmatrix}
0 &1\\
\frac{ab}{z(1-z)} & \frac{(a+b+1)z-c}{z(1-z)}
\end{pmatrix}.
$$
Following a computation in {\sc Maple} using the {\tt dsolve} procedure, the field 
$$K\big({_2F_1}(a,b;c;z),\partial_z\big({_2F_1}(a,b;c;z)\big),z^{1-c}{_2F_1(a-c+1,b-c+1;2-c;z)}, z^{1-c}\partial_z\big({_2F_1(a-c+1,b-c+1;2-c;z)}\big) \big)$$ is a 
Picard--Vessiot field containing a complete set of solutions of~\eqref{eq:hg}. Its transcendence degree over $K$ is $4$, because its differential Galois group over $K$ is $\GL_2$, whose dimension is $4$. 
The classical contiguity relations for $_2F_1$, that is, linear in $K$ expressions of $g({_2F_1})$, $g \in \N^3$, via $_2F_1$ and $\partial_z({_2F_1})$, can be seen in Tannakian terms by observing that the differential module $M$ is isomorphic to the differential modules $T(\sigma_i)(M)$ over $K$ via the gauge transformations $C_i^{-1}AC_i-C_i^{-1}\partial_z(C_i)$, $1\leq i \leq 3$, where
$$
C_1 := \begin{pmatrix}
\frac{c-zb-a-1}{a} & \frac{z(z-1)}{a}\\
b& z-1
\end{pmatrix},\quad
C_2 := \begin{pmatrix}
\frac{c-za-b-1}{b} & \frac{z(z-1)}{b}\\
a& z-1
\end{pmatrix},\quad
C_3 := \begin{pmatrix}
c&z\\
\frac{ab}{1-z}& \frac{z(a+b-c)}{1-z}
\end{pmatrix},
$$
respectively,
which can be found, for instance, using the {\tt dsolve} procedure of {\sc Maple}. 

More generally, (non-linear) relations between solutions of parameterized differential and difference equations and their orbits under the action of a monoid $G$, can be exhibited in the Tannakian terms by comparing the tensor categories generated by $T(g)(M)$, $g \in G$.  Developing general algorithms to attack this problem, including efficient termination criteria, is left for future research (cf. \cite[Section~3.2.1 and Proposition~3.2]{MiOvSi} for the case of differential parameters). See also \cite[Examples~2.2 and 3.2]{OvAAM} for the $q$-difference analogue of the hypergeometric functions, where its isomonodromy properties are explicitly computed.
\end{ex}
\subsection{Semigroup actions on tensor categories}\label{sec:semigrouptensor}
\begin{defi}\label{def:Gtensor} A $G$-$\otimes$-category is an abelian tensor category $\C$ together with an action $T$ of  $G$ on $\C$ such that
\begin{enumerate}
\item for all $g \in G$,  $T(g) : \C \to \C$ is a right-exact tensor functor and
\item  for all $f,g \in G$, $c_{f,g} : T(f)\circ T(g) \to T(fg)$ and $\iota : T(e) \to \id_\C$ are isomorphisms of tensor functors.
\end{enumerate}
\end{defi}
\begin{rem}
The right-exactness (see also Remark~\ref{rem:rigidexact}) appears in Definition~\ref{def:Gtensor} because it is used in our main result,
Theorem~\ref{theo: sTannakianmain}, so that we are able to apply Theorem~\ref{theo: scalar extension for categories} there.
\end{rem}

\begin{prop}\label{prop:310} Let $$G \cong\N^n\times \Zb/{n_1}\Zb\times\ldots\times\Zb/{n_r}\Zb,\quad n_j \geq 1,\ 1\leq j\leq r,$$ with a selected set  $\{a_1,\ldots,a_m\}$, $m=n+r$, of generators corresponding to the decomposition. Then defining a $G$-$\otimes$ category structure on an abelian tensor category $\C$ is equivalent to defining:
\begin{enumerate}
\item right-exact tensor functors $T(a_i) :\C \to \C$, $1\leq i\leq m$, 
\item isomorphisms of tensor functors $i_{a_i,a_j} : T(a_i)\circ T(a_j) \stackrel{\sim}{\longrightarrow} T(a_j)\circ T(a_i)$, $1\leq i,j \leq m$, that satisfy the hexagon axiom~\eqref{eq:i1i2i3}, and
\item isomorphisms of tensor functors $I_j:T(a_{n+j})^{\circ n_j}\to \id_\C$, $1\leq j\leq r$ that satisfy~\eqref{eq:finordercond}.
\end{enumerate}
\end{prop}
\begin{proof} This follows from Theorem~\ref{thm:main0} and the discussion that directly precedes it.
\end{proof}

\begin{cor}Moreover, we have:
\begin{enumerate}
\item If $m=n=1$, that is $G \cong \N$, then a defining $G$-$\otimes$ category structure on an abelian tensor category $\C$ is equivalent to defining a right-exact tensor functor $T(a_1) :\C \to \C$.
\item If $m=2$, then (as noted and used in Example~\ref{ex:commuteandnot}) the hexagon axiom~\eqref{eq:i1i2i3} is not needed, because it becomes non-trivial only for $m\geq 3$.
\end{enumerate}
\end{cor}

If $\C$ is a $G$-$\otimes$-category, then $R:=\End(\1)$ is naturally a difference ring via $$T(g)\colon \End(\1)\to\End(T(g)(\1))\simeq\End(\1),\quad g\in G.$$
The latter isomorphism is derived from the uniqueness of the identity object and the fact that a tensor functor respects identity objects.
Note that, for all $g \in G$, $T(g)\colon\C\to \C$ is {\it $T(g)$-linear}. That is, $$T(g)(r\varphi)=T(g)(r)T(g)(\varphi)$$ for every morphism $\varphi$ in $\C$ and $r\in R$.

\begin{defi}
Let $R$ be $G$-ring. An \emph{$R$-linear $G$-$\otimes$-category} is a $G$-$\otimes$-category that is $R$-linear and such that the canonical ring morphism $l : R\to \End(\1)$ is a morphism of $G$-rings.   
An \emph{$R$-linear $G$-$\otimes$-category} is said to be over $R$ if $l$ is an isomorphism of $G$-rings.
\end{defi}

The following is the prototypical example of a difference tensor category.

\begin{ex}\label{ex:GModR}
Let $R$ be a $G$-ring. The category $\mod_R$ of $R$-modules is naturally a $G$-$\otimes$-category:
\begin{itemize}
\item The tensor product is the usual tensor product of $R$-modules. 
\item The right-exact tensor functor $$T(g)\colon \mod_R\to\mod_R,\ M\rightsquigarrow {\hs M}$$ is given by base extension via $g\colon R\to R$. That is, $$T(g)(M)={\hs M}=M\otimes_R R.$$ The $R$-module structure of $\hs M$ comes from the right factor.  So, explicitly for $m\in M$ and $r,s\in R$ we have $$s\cdot(m\otimes r)=m\otimes sr \ \text{ and } \ sm\otimes r=m\otimes g(s)r.$$ 
\item For all $g$, $h \in G$, and an $R$-module $M$, $$c_{g,h} : T(g)T(h)(M) \to T(gh)(M),\quad m\otimes r\otimes s \mapsto m\otimes g(r)s,\quad m\in M,\ r,s \in R.$$
\item The functorial isomorphism, which is part of the data of a tensor functor, is the natural one: $${\hs M}\otimes {\hs N}\simeq {\hs(M\otimes N)}.$$
\item
The identity object $(\1, e)$ is the free $R$-module $\1=Rb$ of rank one with basis $b$ together with $e\colon \1\to\1\otimes \1$ determined by $e(b)=b\otimes b$.
\end{itemize}
Note that, by identifying $R$ with $\End(\1)$, we recover the original $T(g)\colon R\to R$ from $T(g)\colon \End(\1)\to \End(\1)$.
\end{ex}

In what follows, we will always consider the category of modules over a $G$-ring with the above described $G$-$\otimes$-structure. In particular, if $k$ is a $G$-field, then $\vect_k$ is naturally a $G$-$\otimes$-category (over $k$).

\begin{defi}
Let $\C$ and $\D$ be $G$-$\otimes$-categories 
via $T_\C$ and $T_\D$, respectively. A \emph{$G$-$\otimes$-functor} $\C\to\D$ is a pair $(F,\alpha)$ comprising a 
tensor functor $F\colon \C\to\D$ and a set of  isomorphisms of 
tensor functors $$\alpha = \left\{\alpha_g\colon F\circ T_\C(g)\rightarrow T_\D(g)\circ F\colon \C\to\D\:|\: g \in G\right\}$$
such that, for all $f, g\in G$, the following diagram is commutative 
\begin{equation}\label{eq:Gfunctor}
\begin{CD}
F\circ T_\C(f)\circ T_\C(g)@>\id\circ {c_\C}_{f,g}>> F\circ T_\C(fg)\\
@VV\left(\id\circ\alpha_g\right)\left(\alpha_f\circ\id\right)V@VV\alpha_{fg}V\\
T_\D(f)\circ T_\D(g)\circ F@>{c_\D}_{f,g}\circ\id>>T_\D(fg)\circ F
\end{CD}
\end{equation}
\end{defi}

\begin{prop}  \label{prop: good def of Gtensor functor}Let $$G \cong \N^n\times \Zb/{n_1}\Zb\times\ldots\times\Zb/{n_r}\Zb,\quad n_j \geq 1,\ 1\leq j\leq r,$$ with a selected set  $\{a_1,\ldots,a_m\}$, $m=n+r$, of generators corresponding to the decomposition. Then the set $\alpha$ can be replaced with its finite subset $$\left\{\alpha_{a_i}\colon F\circ T_\C(a_i)\rightarrow T_\D(a_i)\circ F\colon \C\to\D\:|\: 1\leq i\leq m\right\}$$ and the former of the sets of commutative diagrams in~\eqref{eq:Gfunctor} can be replaced with the following finite set of commutative diagrams: 
$$
\xy
(0,0)*+{F\circ T_\C(a_i)\circ T_\C(a_j)}="a", 
(20,18)*+{ F\circ T_\C(a_j)\circ T_\C(a_i)}="b",
(60,18)*+{ T_\D(a_j)\circ F \circ T_\C(a_i)}="c", 
(80,0)*+{T_\D(a_j)\circ T_\D(a_i)\circ F}="d",
(60,-18)*+{T_\D(a_i)\circ T_\D(a_j)\circ F}="e", 
(20,-18)*+{T_\D(a_i)\circ F \circ T_\C(a_j)}="f"
\ar^{\id\circ {i_\C}_{a_i,a_j}} @{->} "a";"b"
\ar^{\alpha_{a_j}\circ\id} @{->} "b";"c"
\ar^{\id\circ\alpha_{a_i}} @{->} "c";"d"
\ar_{\alpha_{a_i}\circ\id} @{->} "a";"f"
\ar^{\id\circ\alpha_{a_j}} @{->} "f";"e"
\ar_{{i_\D}_{a_i,a_j}\circ\id} @{->} "e";"d"
\endxy \quad\quad i>j,\ 1\leq i,j\leq m.
$$
and, for all $i$, $n < i \le m$,
$$
\begin{CD}
F\circ T_\C(a_i)^{n_i}@>\id\circ {I_i}_{\C}>> F\\
@V\alpha_{a_i}^{n_i}VV@|\\
T_\D(a_i)^{n_i}\circ F@>{I_i}_\D\circ\id>>F.
\end{CD}
$$
\end{prop}
\begin{proof} This can be proven as in Proposition~\ref{prop:310}.
\end{proof}

\begin{ex}
Let $R$ be a $G$-ring and $S$ an $R$-$G$-algebra. Then $\mod_R\to \mod_S,\ M\rightsquigarrow M\otimes_R S$ together with the functorial isomorphisms
$$\alpha_{g,M}\colon {\hs M}\otimes_R S=(M\otimes_R R)\otimes_R S\simeq M\otimes_R S\simeq (M\otimes_R S)\otimes_S S={\hs (M\otimes_R S)},\quad g\in G,$$ derived from the commutativity of
\[\xymatrix{
 R \ar[r] \ar_{g}[d] & S \ar^{g}[d] \\
 R \ar[r] & S
}
\]
is a $G$-$\otimes$-functor.
\end{ex}

The composition of $G$-$\otimes$-functors is a $G$-$\otimes$-functor in a natural way.
\begin{defi}
Let $(F,\alpha)$, $\big(F',\alpha'\big)\colon \C\to \D$ be $G$-$\otimes$-functors. A \emph{morphism of $G$-$\otimes$-functors} $(F,\alpha)\to \big(F',\alpha'\big)$ is a morphisms of $\otimes$-functors $\beta \colon F\to F'$ such that the diagram
 \begin{equation}\label{eq:morfung}
\begin{CD}
 F\circ T_\C(g) @>\beta\circ\id>> F'\circ T_\C(g)\\
 @V{\alpha_g}VV@V{\alpha'_g}VV \\
 T_\D(g)\circ F @>\id\circ\beta>>T_\D(g)\circ F'
\end{CD}
\end{equation}
commutes for all $g \in G$.
\end{defi}

\begin{lemma} \label{lemma: good def of morphism of Gtensorfunct}
Let $$G=\N^n\times \Zb/{n_1}\Zb\times\ldots\times\Zb/{n_r}\Zb,\quad n_j \geq 1,\ 1\leq j\leq r,$$ with a selected set $\{a_1,\ldots,a_m\}$, $m=n+r$, of generators corresponding to the decomposition. Let $(F,\alpha)$, $\big(F',\alpha'\big)\colon \C\to \D$ be $G$-$\otimes$-functors. Then a morphisms of $\otimes$-functors $\beta \colon F\to F'$ is a morphism of $G$-$\otimes$-functors if and only if
\begin{equation}\label{eq:morfunij}
\begin{CD}
 F\left(T(a_i)(X)\right) @>\beta_{T(a_i)(X)}>> F'\left(T(a_i)(X)\right)\\
 @V{\alpha_{a_i}}_XVV@V{\alpha'_{a_i}}_XVV \\
 T(a_i)(F(X))@>T(a_i)(\beta_X)>>T(a_i)\left(F'(X)\right)
\end{CD}
\end{equation}
commutes for every object $X$ of $\C$ and $i=1,\ldots,m$.
\end{lemma}
\begin{proof}
For all $g = a_1^{d_1}\cdot\ldots\cdot a_m^{d_m}$ and $X \in \Ob(\C)$, since $\beta$ is a morphism of functors $F\to F'$ and by~\eqref{eq:morfunij}, the following diagram is commutative:
$$
\begin{CD}
F(T(g)(X))@>F(c_X)>> F\big(T(a_1)^{\circ d_1}\circ\ldots\circ T(a_m)^{\circ d_m}(X)\big)@>\alpha_X>>T(a_1)^{\circ d_1}\circ\ldots\circ T(a_m)^{\circ d_m}F(X)\\
@V\beta_{T(g)(X)}VV@V\beta_{T(a_1)^{\circ d_1}\circ\ldots\circ T(a_m)^{\circ d_m}(X)}VV@VT(a_1)^{\circ d_1}\circ\ldots\circ T(a_m)^{\circ d_m}(\beta_X)VV\\
F'(T(g)(X))@>F'(c_X)>> F'\big(T(a_1)^{\circ d_1}\circ\ldots\circ T(a_m)^{\circ d_m}(X)\big)@>\alpha'_X>>T(a_1)^{\circ d_1}\circ\ldots\circ T(a_m)^{\circ d_m}F'(X)
\end{CD}
$$
where $c$ is the appropriate isomorphism of functors $T(g) \to T(a_1)^{\circ d_1}\circ\ldots\circ T(a_m)^{\circ d_m}$ obtained as a composition of various $c_{.,.}$; similarly for $\alpha_X$ and $\alpha'_X$. Commutativity of~\eqref{eq:morfunij} now follows from an iterative application of~\eqref{eq:Gfunctor}.
\end{proof}

\subsection{Semigroup actions on Tannakian categories}\label{sec:semigrouptannakian}

\begin{defi}\label{def:nGTcat}
Let $k$ be a $G$-field. A \emph{neutral $G$-Tannakian category} over $k$ is a $G$-$\otimes$-category $\C$ over $k$ that is rigid (as a tensor category) and such that there exists a \emph{$G$-fibre functor} $\C\to\vect_k$, i.e., a $G$-$\otimes$-functor $(F,\alpha)$ with $F$ exact, faithful and $k$-linear.
\end{defi}

\begin{rem}\label{rem:rigidexact}
If $F\colon\C\to\D$ is a right-exact tensor functor between abelian tensor categories and $\C$ is rigid, then $F$ is exact \cite[Lemma~2.2.3]{Stalder:ScalarExtensionofAbelianandTannakianCategories}. Therefore, for all $g\in G$, $T(g)\colon \C\to\C$ is exact for every neutral $G$-Tannakian category $\C$.
\end{rem}

\begin{ex} \label{ex: repG is sTannakian}
Let $k$ be a $G$-field and $H$ a group $k$-$G$-scheme. The category $\rep(H)$ of representations of $H$ is a neutral $G$-Tannakian category over $k$ in a natural way:
\begin{itemize}
\item The tensor product and dual are as described in Section~\ref{sec:constructions}.
\item The right-exact tensor functors
$$T(g)\colon \rep(H)\to \rep(H),\ V\rightsquigarrow{\hs V},\quad g\in G,$$ are also described in Section \ref{subsec: Basic constructions}. 
\item For all $g$, $h \in G$, and $V \in\Ob(\rep(H))$, $$c_{g,h} : T(g)T(h)(V) \to T(gh)(V),\quad v\otimes r\otimes s \mapsto v\otimes g(r)s,\quad v\in V,\ r,s \in k.$$
\item The $G$-$\otimes$-functor $$\omega\colon \rep(H)\to \vect_k$$ that forgets the action of $H$ is a $G$-fibre functor for $\rep(H)$.
\end{itemize}
\end{ex}

Theorem \ref{theo: sTannakianmain} below asserts that the above example is ``essentially'' the only example of a neutral $G$-Tannakian category. However, there are natural examples of neutral $G$-Tannakian categories for which the determination of the corresponding group $G$-scheme is a highly non--trivial problem:

\begin{ex}\label{ex:diffmod} We will describe the $G$-$\otimes$-category of {\it differential modules}. Let $K$ be a $G$-field and a $\partial$-field (that is, $\partial : K\to K$ is a derivation) such that, for all $g \in G$, $T(g) : K\to K$ commutes with $\partial$. 
As in~\cite[Definition~1.6 and~Section~2.2]{SingerPut:differential}, 
\begin{itemize}
\item 
the objects are finite-dimensional $K$-vector spaces $M$ with an additive map $\partial :M\to M$ satisfying $\partial(am) = \partial(a)m+a\partial(m)$, $a \in K$, $m\in M$; 
\item the morphisms are $K$-linear maps that commute with $\partial$; the tensor structure is as in the vector spaces, with $\partial(m\otimes n)=\partial(m)\otimes n + m\otimes\partial(n)$, $m \in M$, $n \in N$. 
\item The $G$-action is given as in Example~\ref{ex:GModR}.
\end{itemize}
Similarly to \cite{SingerPut:differential}, this is a $G$-$\otimes$-category over the $G$-field $K^\partial=\{a\in K\:|\: \partial(a) = 0\}\cong\End(\1)$.
\end{ex}

Let $k$ be a $G$-field, $\C$ a neutral $G$-Tannakian category over $k$ and $\omega\colon\C\to\vect_k$ a $G$-fibre functor. For every $k$-$G$-algebra $R$, composing $\omega$ with the $G$-$\otimes$-functor $$\vect_k\to\mod_R,\ V\rightsquigarrow V\otimes_kR,$$ yields a $G$-$\otimes$-functor
$$\omega\otimes R\colon\C\to\mod_R.$$ Let $\underline{\operatorname{Aut}}^{G,\otimes}(\omega)(R)$ denote the group of all automorphisms of $\omega\otimes R$ (i.e., invertible morphisms $\omega\otimes R\to\omega\otimes R$ of $G$-$\otimes$-functors). Then
$\underline{\operatorname{Aut}}^{G,\otimes}(\omega)$ is naturally a functor from $k$-$G$-$\alg$ to $\groups$.

If $\C=\rep(H)$ and $\omega$ are as in Example~\ref{ex: repG is sTannakian}, we have a canonical morphism $$H\to\underline{\operatorname{Aut}}^{G,\otimes}(\omega)$$ of group functors on $k$-$G$-$\alg$. (The statement that $h\in H(R)$, when considered as a morphism of functors $h\colon\omega\otimes R\to \omega\otimes R$, respects $G$ is precisely identity (\ref{eqn: RepG is scategory}).)

For a $k$-$G$-algebra $R$ let $R^\sharp$ denote the $k$-algebra obtained from $R$ by forgetting the $G$-action. Similarly, for a group $k$-$G$-scheme $H$, let $H^\sharp$ denote the group scheme obtained from $H$, by forgetting the $G$-action, i.e., $H^\sharp $ is the affine group scheme represented by the Hopf algebra $k\{H\}^\sharp$.
\begin{prop} \label{prop: stannakaconsistent}
Let $k$ be a $G$-field, $H$ a group $k$-$G$-scheme, and $\omega\colon \rep(H)\to \vect_k$ the forgetful $G$-$\otimes$-functor. Then the canonical morphism $$H\to \underline{\operatorname{Aut}}^{G,\otimes}(\omega)$$ is an isomorphism.
\end{prop}
\begin{proof}
Let $R$ be a $k$-$G$-algebra. By forgetting the $G$-structure, we can interpret $\omega$ as a fibre functor for a Tannakian category. Then \cite[Proposition~2.8]{DeligneMilne:TannakianCategories} says that the natural map $$H^\sharp\big(R^\sharp\big)\to \underline{\operatorname{Aut}}^{\otimes}(\omega)(R)$$ is bijective. It therefore suffices to see that, under this bijection, $H(R)\subset H^\sharp(R^\sharp)$ corresponds to
$$\underline{\operatorname{Aut}}^{G,\otimes}(\omega)(R)\subset \underline{\operatorname{Aut}}^{\otimes}(\omega)(R).$$
Thus, we have to show that, for an isomorphism of $G$-$\otimes$-functors $\beta\colon\omega\otimes R\to\omega\otimes R$, the corresponding morphism $$h\in\Hom_k\big(k\{H\}^\sharp,R^\sharp\big)=H^\sharp\big(R^\sharp\big)$$ is a morphism of difference rings. Let $\varphi\in k\{H\}$. We have to show that $$h(g(\varphi))=g(h(\varphi)),\quad g \in G.$$
Using Sweedler's notation, we may write 
\begin{equation}\label{eq:coaction}
\Delta(\varphi)=\sum \varphi_{(1)}\otimes \varphi_{(2)}\in V\otimes_k k\{H\}.
\end{equation}
Then $V := \Span_k\big\{\varphi_{(1)}\big\}$ is a finite-dimensional $H$-stable $k$-subspace of $k\{H\}$ containing $\varphi$, as $(\id\otimes\varepsilon)\circ\Delta(\varphi)=1\otimes \varphi$. By assumption, for all $g\in G$,
\begin{equation} \label{eqn: betasfunctor}
\xymatrix{
{\hs V}\otimes_k R \ar^{\beta_{{\hs V}}}[r] \ar_{\simeq}[d] & {\hs V}\otimes_k R \ar^{\simeq}[d] \\
{\hs(V\otimes_k R)} \ar^{g(\beta_V)}[r] & {\hs(V\otimes_k R)}
}
\end{equation}
commutes. By~\eqref{eq:coaction}, $$\beta_V(\varphi\otimes 1)=h(\varphi\otimes 1)=\sum \varphi_{(1)}\otimes h(\varphi_{(2)})\in V\otimes_k R.$$ Chasing $(\varphi\otimes 1)\otimes 1\in {\hs V}\otimes_k R$ through diagram (\ref{eqn: betasfunctor}), we see  that
$$\sum \varphi_{(1)}\otimes h(g(\varphi_{(2)}))=\sum \varphi_{(1)}\otimes g(h(\varphi_{(2)}))\in V\otimes_k R,$$ where the latter tensor product is formed by using $k\xrightarrow{g}k\to R$ on the right-hand side. Applying the counit $\varepsilon\colon k\{H\}\to k$ to this identity, we conclude that
$$\sum g(\varepsilon(\varphi_{(1)}))h(g(\varphi_{(2)}))=h\big(\sum g(\varepsilon(\varphi_{(1)})g(\varphi_{(2)}))\big)=h(g(\varphi))$$ and
$$\sum g(\varepsilon(\varphi_{(1)}))g(h(\varphi_{(2)}))=g\big(\sum\varepsilon(\varphi_{(1)})h(\varphi_{(2)})\big)=g(h(\varphi))$$ are equal. So, as claimed, $h$ is a morphism of $G$-rings.
\end{proof}

\begin{theo} \label{theo: sTannakianmain}
Let $k$ be a $G$-field and $(\C,\omega)$ a neutral $G$-Tannakian category over $k$. Then $H=\underline{\operatorname{Aut}}^{G,\otimes}(\omega)$ is a group $k$-$G$-scheme, and $\omega$ induces an equivalence of $G$-$\otimes$-categories over $k$ between $\C$ and $\rep(H)$.
\end{theo}
In the proof of Theorem \ref{theo: sTannakianmain}, we will use extension of scalars for categories. All we need to know about this concept is contained in \cite{Stalder:ScalarExtensionofAbelianandTannakianCategories}; see also \cite[Section~4.1]{GGO} and the references given there.

\begin{theo}[{\cite[Theorem~1.4.1]{Stalder:ScalarExtensionofAbelianandTannakianCategories}}] \label{theo: scalar extension for categories}
Let $K|k$ be a field extension and $(\C,\omega)$ a neutral Tannakian category over $k$. Then there exists a neutral Tannakian category $\big(\C\otimes_k K,\omega_K\big)$ over $K$, together with an exact $k$-linear tensor functor $T_K\colon\C\to \C\otimes_k K$ satisfying the following universal property:

If $\D$ is an abelian tensor category over $K$ and $F\colon\C\to\D$ an exact $k$-linear tensor functor, then there exists an exact $K$-linear tensor functor $E \colon\C\otimes_k K\to \D$ and an isomorphism of tensor functors $\alpha\colon ET_K\simeq F $. If $(E',\alpha')$ is another such pair, there exists a unique isomorphism of tensor functors $E\simeq E'$ such that
\[\xymatrix{
 ET_K \ar[rr] \ar_{\alpha}[rd] & & E'T_K \ar^{\alpha'}[ld] \\
  & F &
}
\]
commutes.
\end{theo}

\begin{rem}
In the situation of Theorem \ref{theo: scalar extension for categories}, if $\omega\colon\C\to\vect_k$ is a fibre functor, then the universal property applied to $\omega\otimes K\colon\C\to \vect_K$ yields a fibre functor $$\omega_K\colon \C\otimes_k K\to \vect_K.$$ The functor $\omega_K$ is faithful since any non--zero exact tensor functor on a rigid abelian tensor category is faithful  \cite[Lemma 2.2.4]{Stalder:ScalarExtensionofAbelianandTannakianCategories}, \cite[Proposition~1.19]{DeligneMilne:TannakianCategories}.
\end{rem}
\begin{cor} \label{cor: baseextension for omega}
The affine group $K$-scheme $\underline{\operatorname{Aut}}^{\otimes}(\omega_K)$ is obtained from the affine group $k$-scheme $\underline{\operatorname{Aut}}^{\otimes}(\omega)$ by base extension from $k$ to $K$.
\end{cor}
\begin{proof}
This follows from \cite[Theorem~2.11]{DeligneMilne:TannakianCategories} and \cite[Theorem~3.1.7(d)]{Stalder:ScalarExtensionofAbelianandTannakianCategories}.
\end{proof}

\begin{proof}[Proof of Theorem \ref{theo: sTannakianmain}] This is an analogue of \cite[Proposition~4.25]{GGO}. Let $\C^\sharp$ denote the tensor category obtained from $\C$ by forgetting $G$. Similarly, let $$\omega^\sharp\colon\C^\sharp\to\vect_k$$ denote the tensor functor obtained from $\omega$ by forgetting the difference structure. Then $\C^\sharp$ is a neutral Tannakian category over $k$ with fibre functor $\omega^\sharp$. We know from Theorem \ref{theo: Tannakianmain} that $$H^\sharp=\underline{\operatorname{Aut}}^\otimes(\omega^\sharp)$$ is an affine group scheme over $k$.
The crucial step now is to use the difference structure on $\C$ to put a difference structure on $H^\sharp$.

For an extension of fields $g\colon k\to K$, a $K$-linear category $\D$ is naturally a $k$-linear category, which we denote by $_g\D$.
The universal property from Theorem ~\ref{theo: scalar extension for categories} may be expressed by
\[\xymatrix{
 \C\ar^{T_{K}}[rr] \ar_{F}[rd] & & {\C\otimes_k K} \ar@{.>}^{\wtilde{F}}[ld] \\
  & {_gD=D} &
}
\]
Here $F\colon\C\to_g\D$ is $k$-linear and $\wtilde{F}\colon\C\otimes_k K\to D$ is $K$-linear.
Now let $g\in G$ and apply the above with $K=k$, $\C=\C^\sharp$, $\D=\C^\sharp$, $F=T(g)$, and
$$\hs(\C^\sharp)=\C\otimes_kK=\C^\sharp\otimes_k k,$$
where the tensor product uses $g\colon k\to k=K$ on the right-hand side.
We obtain the universal property
\[\xymatrix{
 \C^\sharp \ar^{T_g}[rr] \ar_{T(g)}[rd] & & {\hs(\C^\sharp)} \ar@{.>}^{\wtilde{T(g)}}[ld] \\
  & {_g(\C^\sharp)=\C^\sharp} &
}
\]
So there exists a right--exact $k$-linear tensor functor $\wtilde{T(g)}\colon{\hs(\C^\sharp)}\to{\C^\sharp}$ (which is not necessarily an equivalence of categories) and an isomorphism of tensor functors $$\alpha_g\colon \wtilde{T(g)}T_g\simeq T(g).$$ 
Hence, for all $g,h \in G$, we have isomorphisms of tensor functors:
$$
\alpha_{g,h} : \widetilde{T(g)}T_g\widetilde{T(h)}T_{k_h} \xrightarrow{\alpha_g\circ\alpha_h} T(g)T(h)\xrightarrow{c_{g,h}}T(gh)\xrightarrow{\alpha_{gh}^{-1}}\widetilde{T(gh)}\circ T_{gh}.
$$
As in Corollary \ref{cor: baseextension for omega}, $$T(g)\omega^\sharp\colon\C^\sharp\to\vect_k$$ induces a fibre functor
$${\hs(\omega^\sharp)}\colon {\hs(\C^\sharp)}\to\vect_k$$ and an isomorphism of tensor functors $$\beta_g :  {\hs(\omega^\sharp)}T_k  \simeq T(g)\omega^\sharp.$$
We have
\begin{equation}\label{eq:deltag}
\delta_g : \omega^\sharp\wtilde{T(g)}T_k\xrightarrow{\id_{\omega^\sharp}\circ\alpha_g}\omega^\sharp T(g)\xrightarrow{{\alpha_\omega}_g} T(g)\omega^\sharp,
\end{equation}
where  the given isomorphisms of tensor functors $${\alpha_\omega}_g : \omega\circ T(g)\to T(g)\circ\omega,\quad g \in G,$$
satisfy~\eqref{eq:Gfunctor} by the hypothesis of the theorem.
So, there exists a unique isomorphism of tensor functors
$$\gamma_g : {\hs(\omega^\sharp)}\simeq \omega^\sharp\wtilde{T(g)}$$ making
\[\xymatrix{
 {\hs(\omega^\sharp)}T_k \ar^{\gamma_g\circ\id_{T_k}}[rr] \ar_{\beta_g}[rd] & & \omega^\sharp\wtilde{T(g)}T_k \ar^{\delta_g}[ld] \\
  & T(g)\omega^\sharp &
}
\]
commutative.

Let ${\hs(H^\sharp)}$ be the affine group scheme over $k$ obtained from $H^\sharp$ by base extension via $g\colon k\to k$. By Corollaries~\ref{cor: morphism of groups corresponds to functor} and~\ref{cor: baseextension for omega}, the tensor functor $\wtilde{T(g)}\colon{\hs(\C^\sharp)}\to\C^\sharp$ and the isomorphism ${\hs(\omega^\sharp)}\simeq \omega^\sharp\wtilde{T(g)}$ correspond to a morphism of
group $k$-schemes $$m_g : H^\sharp\to{\hs(H^\sharp)}.$$
But to give a morphism of group $k$-schemes $H^\sharp\to{\hs(H^\sharp)}$ is really just the same thing as turning $H^\sharp$ into a group $k$-$g$-scheme $H$, as is best understood from the Hopf--algebraic point of view:
$$
M_g : k\big\{H^\sharp\big\}\xrightarrow{\id\otimes 1} k\big\{\hs(H^\sharp)\big\} = k\big\{H^\sharp\big\}\otimes_kk \xrightarrow{m_g^*}  k\big\{H^\sharp\big\}.
$$
Denote $A:= k\big\{H^\sharp\big\}$. For all $g \in G$, since $m_g$ is a morphism of group $k$-schemes, the Hopf algebra structure homomorphisms $\Delta : A\to A\otimes_k A$, $S : A \to A$ and $\varepsilon : A\to k$ commute with $M_g$.
Moreover, we have 
\begin{equation}\label{eq:Mgh}
M_{gh} = M_g\circ M_h,\quad g,h \in G,
\end{equation} by Corollary~\ref{cor: morphism of groups corresponds to functor}, because \[\xymatrix{
\omega^\sharp\widetilde{T(g)}T_g\widetilde{T(h)}T_h \ar^{\id_{\omega^\sharp}\alpha_{g,h}}[rr] \ar_{\gamma_{g,h}}[rd] & & \omega^\sharp \widetilde{T(gh)}T_{gh}  \ar^{\big(\gamma_{gh}\circ\id_{T_{gh}}\big)^{-1}}[ld] \\
  & {^{gh}(\omega^\sharp)T_{gh}} &
}
\]
commutes, where
\begin{align*}
\gamma_{g,h} &: \omega^\sharp\widetilde{T(g)}T_g\widetilde{T(h)}T_h \xrightarrow{\delta_g\id_{\widetilde{T(h)}T_h}} T(g)\omega^\sharp\widetilde{T(h)}T_h\xrightarrow{\id_{T(g)}\delta_h} T(g)T(h)\omega^\sharp\to\\
&\xrightarrow{c_{g,h}\id_{\omega^\sharp}} T(gh)\omega^\sharp\xrightarrow{\beta_{gh}^{-1}} {^{gh}(\omega^\sharp)T_{gh}}.
\end{align*}
That is, to prove~\eqref{eq:Mgh}, remains to show that
\begin{equation*}
\beta_{gh}^{-1}\circ\big(c_{g,h}\id_{\omega^\sharp}\big)\big(\id_{T(g)}\delta_h\big)\big(\delta_g\id_{\widetilde{T(h)} T_h}\big)=\beta_{gh}^{-1}\circ\delta_{gh}\circ\big(\id_{\omega^\sharp}\alpha_{gh}^{-1}\big)\big(\id_{\omega^\sharp}c_{g,h}\big)\big(\id_{\omega^\sharp}\alpha_g\alpha_h\big).
\end{equation*}
By~\eqref{eq:deltag}, the above is equivalent to
\begin{equation*}
\big(c_{g,h}\id_{\omega^\sharp}\big)\big(\id_{T(g)}\delta_h\big)\left(\delta_g\id_{\widetilde{T(h)} T_h}\right)={\alpha_\omega}_{gh}\circ\big(\id_{\omega^\sharp}c_{g,h}\big)\big(\id_{\omega^\sharp}\alpha_g\alpha_h\big),
\end{equation*}
which, by~\eqref{eq:Gfunctor}, is equivalent to
\begin{equation*}
\big(\id_{T(g)}\delta_h\big)\left(\delta_g\id_{\widetilde{T(h)} T_h}\right)=\big(\id_{T(g)}{\alpha_\omega}_h\big)\big({\alpha_\omega}_g\id_{T(h)}\big)\big(\id_{\omega^\sharp}\alpha_g\alpha_h\big),
\end{equation*}
which, by~\eqref{eq:deltag} again, is equivalent to
\begin{equation*}
{\alpha_\omega}_g\alpha_h=\big(\id_{T(g)\omega^\sharp}\alpha_h\big)\left({\alpha_\omega}_g\id_{\widetilde{T(h)} T_h}\right)=\big({\alpha_\omega}_g\id_{T(h)}\big)\left(\id_{\omega^\sharp T(h)}\alpha_h\right)={\alpha_\omega}_g\alpha_h,
\end{equation*}
concluding the argument showing~\eqref{eq:Mgh}.

The situation is summarized in the following diagram:
$$
\xymatrix{
 \C^\sharp \ar^{T_k}[r] \ar[d] \ar@/^2pc/^{T(g)}[rr] & {\hs(\C^\sharp)} \ar^{\wtilde{T(g)}}[r] \ar[d] & \C^\sharp \ar[d]   \\
 \rep(H^\sharp) \ar[r] \ar@/_2pc/_{T(g)}[rr]& \rep\big({\hs(H^\sharp)}\big) \ar[r] &  \rep(H^\sharp)
 }
$$
Since the two squares commute, up to an isomorphism of tensor functors, the outer rectangle also commutes up to an isomorphism of tensor functors. This shows that $\C\to\rep(H)$ is naturally a $G$-$\otimes$-functor and therefore an equivalence of $G$-$\otimes$-categories. It is then clear from Proposition \ref{prop: stannakaconsistent} that $H$ is isomorphic to $\underline{\operatorname{Aut}}^{G,\otimes}(\omega)$.
\end{proof}

\begin{rem}
Let $$G=\N^n\times \Zb/{n_1}\Zb\times\ldots\times\Zb/{n_r}\Zb,\quad n_j \geq 1,\ 1\leq j\leq r,$$ with a selected set $\{a_1,\ldots,a_m\}$, $m=n+r$, of generators corresponding to the decomposition.
Then Theorem \ref{theo: sTannakianmain} remains verbatim valid if we replace the definition of $G$-$\otimes$-tensor category as in Proposition \ref{prop:310}, the definition of $G$-$\otimes$-functor as in Proposition \ref{prop: good def of Gtensor functor} and the definition of morphism of $G$-$\otimes$-functors as in Lemma \ref{lemma: good def of morphism of Gtensorfunct}.
\end{rem}

\subsection{More on representations}\label{sec:extra}
\subsubsection{Explicit formula for semigroup action}\label{sec:explicitaction}
More explicitly, to obtain the $G$-Hopf algebra representing $\underline{\operatorname{Aut}}^{G,\otimes}(\omega)$, similarly to \cite[Section~6.3]{OvchRecoverGroup} and \cite[pp.~370--371]{GGO}, one can take the Hopf algebra $A$ that represents $\underline{\operatorname{Aut}}^{\otimes}(\omega)$: 
$$
A = \bigoplus_{V\in \Ob(\C)} \omega(V)\otimes_k \omega(V)^\vee\Big/R,
$$
where $R$ is the $k$-subspace spanned by
\begin{align*}
\left\{\left(\id\otimes \omega(\phi)^\vee - \omega(\phi)\otimes
\id\right)(z)\:\big|\ V,\:W \in \Ob(\C),\ \phi\in
\MorC(V,W),\ z \in \omega(V)\otimes \omega(W)^\vee\right\},
\end{align*}
and define the action of $G$ on $A$ as follows.
For $V \in \Ob(\C)$, let $v \in \omega(V)$ and $u\in\omega(V)^\vee$. For all $g\in G$, we  define $$T(g)(v\otimes u) \in \omega\big(T(g)(V)\big)\otimes_k\omega\big(T(g)(V)\big)^\vee$$ by:
$$
T(g)(v\otimes u) := (v\otimes 1)\otimes T(g)(u),\quad T(g)(u)(w\otimes a) := aT(g)(u(w)),\ \ w \in \omega(V),\ a\in k.
$$
For $A$ defined  as in \cite[pp.~370--371]{GGO}, one uses the same formula but conjugated by the isomorphism
$$
\varphi : \eta(V)\otimes_k\omega(V)^\vee \to \Hom_k(\omega(V),\eta(V)),\ \ \varphi(v\otimes u)(w) := u(w)v,\ \ v \in \eta(V),\ u\in\omega(V)^\vee,\ w\in\omega(V),
$$
where, for our purposes, $\eta = \omega$. 

\subsubsection{Characterization of  difference algebraic groups}
In this section, we will show how to recognize categories of representation of  $G$-algebraic groups among those of group $G$-schemes.

 Let $G = \langle S\:\: |\: R\rangle$ be a presentation with generators and relations. Recall that, for all $g \in G$, $l_{R,S}(g)$ is defined to be the length of a shortest presentation of $g$ as a product of the generators. For all $f \in k\{y_1,\ldots,y_n\}_G$, we define
$$
\ord_{R,S}(f) := \max_{g(y_i) \text{ appears in } f} l_{R,S}(g).
$$
For simplicity, in what follows, we assume that $R$ and $S$ are fixed and drop the subscript $R,S$ from $\ord$.

\begin{defi} We say that an object $V$ of a $G$-$\otimes$-category $\C$ is a $G$-$\otimes$-generator of $\C$ if the set of objects $\{T(g)(V)\:|\: g\in G\}$ generates $\C$ as a tensor category.
\end{defi}

A representation $\f\colon H\to\Gl(V)$ is called {\it faithful} if $\f^*\colon k\{\Gl(V)\}\to k\{H\}$ is surjective.
\begin{theo}\label{thm:fingen}
Let $H$ be a $G$-algebraic group. Then every faithful representations of $H$ $G$-$\otimes$-generates $\rep(H)$.
\end{theo}
\begin{proof}  This proof closely follows the proof of \cite[Proposition~1]{OvchRecoverGroup}. Let $U$ be an $A := k\{H\}$-comodule.
By \cite[Lemma~3.5]{Waterhouse:IntroductiontoAffineGroupSchemes},
 $U$ is an $A$-subcomodule of
$U\otimes_k A \cong A^m$, $\rho_{U\otimes A} := \id_U\otimes\Delta$.
The canonical projections $\pi_i : A^m \to A$
 are $H$-equivariant  (with respect to the comultiplication $\Delta : A\to A\otimes A$).
Since $U \subset A^m$, we have
$$
U \subset \bigoplus_{i=1}^m \pi_i(U),
$$
and each $\pi_i(U)$ is an $A$-comodule.
Let $(V,\f)$ be a faithful representation of $H$ and fix a basis $v_1,\ldots,v_n$ of $V$. Let $$\pi=\f^*: B := k\big\{x_{11},\ldots,x_{nn},1/\det\big\}_G \to A$$ be the corresponding surjection of $k$-$G$-Hopf algebras.
Since $\pi_i\left(U\right)$ is a finite-dimensional $A$-subcomodule of $A$,
there exist $r, s, p \in \Zb_{\ge 0}$ such that $\pi_i\left(U\right)$ is contained
in $\pi(L_{r,s,p}),$ where
$$L_{r,s,p} := (1/\det)^r\big\{f(x_{ij})\:|\:\deg(f)\le s, \ord(f) \le p\big\}.$$
The comultiplication of $B$ is given by, 
for all $i,j$, $1\le i,j\le n$, 
$$
\Delta\big(x_{ij,g}\big) = \sum_{l=1}^n x_{il,g}\otimes x_{lj,g},\quad  g \in G,
$$ and $L_{r,s,p}$ is a $B$-subcomodule of $B$, because
$$
\Delta\big(x_{ij}x_{pq}\big) = \sum_{l,r = 1}^nx_{il}x_{pr}\otimes x_{lj}x_{rq}\quad\text{and}\quad \ord(f_1f_2) = \max\big\{\ord(f_1),\ord(f_2)\big\}.$$
Hence, $L_{r,s,p}$
is also an $A$-subcomodule of $B$. Therefore, each $\pi_i(U)$ is a
subquotient of some $L_{r,s,p}$. Thus, we only need to show how
to construct these $L_{r,s,p}$ from $V$. 
For each $i,$ $1 \le i \le n$, the map $\varphi_i : V\to B$, $v_j \mapsto x_{ij}$ is
$\GL_n$ (hence, $H$)-equivariant, because
$$
(\varphi_i\otimes\id)\big(\rho_V(v_j)\big) =(\varphi_i\otimes\id)\left(\sum_{l=1}^nv_l\otimes x_{lj}\right)
= \sum_{l=1}^nx_{il}\otimes x_{lj} = \Delta(x_{ij})
= \rho_B\big(\varphi_i(v_j)\big).
$$ 
Every $f \in L_{0,1,p}$ is of the form
$$
f = \sum_{i,j=1}^n\sum_{g \in S_f}c_{ij}x_{ij,g},\quad c_{ij} \in k,
$$
for some finite $S_f \subset G$ such that, for all $h \in S_f$, $\ord(h) \le p$.
As it has been noticed above, this space is an
$A$-subcomodule of $B$. The map $(\varphi_1,\ldots,\varphi_n)$ induces
$$\left(T_p(V)\right)^n\cong L_{0,1,p},\quad T_p(V):=\bigoplus_{g\in G \atop l_{R,S}(g)\leq p}{T(g)(V)},$$ as $A$-comodules, not necessarily finite-dimensional.
Hence, one can construct
$L_{0,1,p}$.

Let $s \in \Zb_{\ge 2}$. The $A$-comodule $L_{0,s,p}$ is
the quotient of $\left(L_{0,1,p}\right)^{\otimes s}$ by the symmetric relations. So, we have
all $L_{0,s,p}$. Let now $s = n = \dim_k V$. Then the one-dimensional
representation $\det : H \to \Gl_1$ with $h \mapsto \det(h)$ is in $L_{0,n,p}$.
For $f \in k^\vee$, we have
$$
\det(h)(f)(x) = f(x/\det(h))=\frac{1}{\det(h)}f(x).
$$
Thus,
$$
L_{r,s,p} = ({\det}^*)^{\otimes r}\otimes L_{0,s,p},
$$
which
is what we wanted to construct.
\end{proof}

\begin{cor}
Let $H$ be a group $k$-$G$-scheme. Then $H$ is $G$-algebraic if and only if $\rep(H)$ has a $G$-$\otimes$-generator.
\end{cor}
\begin{proof}
This follows from Theorem~\ref{thm:fingen} using \cite[Proposition~A.2]{GGO} and Section~\ref{sec:explicitaction}.
\end{proof}

\section*{Acknowledgments} The authors are grateful to S. Gorchinskiy for his helpful comments about the final draft of the paper and suggestions to consider the approach from~\cite{Deligne1997}, to separately treat the cases of free products of semigroups, in which the action of the product is fully determined by the action of each of the components of the product, and free finitely generated abelian semigroups, and, for the latter case, to use the system of isomorphisms that correspond to the commutativity condition, which is more convenient for the applications than the set of conditions that the authors used originally. The authors are also grateful to M.F. Singer for his important remarks.
\setlength{\bibsep}{1.3pt}
\bibliographystyle{abbrvnat}
\small
\bibliography{bibdata}

\end{document}